\documentclass[psamsfonts]{amsart}
\usepackage{
graphicx,
fancyhdr,
amsmath,
amsfonts,
dsfont,
amsthm,
amssymb,
mathrsfs,
hyperref,
xcolor,
ulem,
cancel,
enumitem,
cleveref,
}

\newcommand{\E}{\mathbb{E}}
\newcommand{\R}{\mathbb{R}}

\newcommand{\Z}{\mathbb{Z}}

\newcommand{\Q}{\mathbb{Q}}
\newcommand{\N}{\mathbb{N}}

\newcommand{\1}{\mathds{1}}

\newcommand{\Prb}{\mathbb{P}}

\newcommand{\SP}{\mathcal{S}}
\newcommand{\Cone}{\mathcal{C}}

\newcommand{\Length}{\Theta}
\newcommand{\YField}{Y}

\newcommand{\Fc}{\mathcal{F}}
\newcommand{\Ball}{\mathsf{B}}
\newcommand{\rball}{E}
\newcommand{\Qq}{\mathsf{Q}}
\newcommand{\Leb}{\mathrm{Leb}}

\newcommand{\Paths}{\mathcal{P}}
\newcommand{\NicePaths}{\mathcal{Q}}

\newcommand{\eps}{\varepsilon}
\newcommand{\strdist}{\rho}
\newcommand{\as}{\text{\rm a.s.}}
\newcommand{\Amin}{\mathcal{A}}
\newcommand{\Bmin}{\mathcal{B}}

\newcommand{\Icube}{\mathsf{I}}
\newcommand{\Jcube}{\mathsf{J}}

\newtheorem{proposition}{Proposition}[section]

\newtheorem{lemma}{Lemma}[section]
\newtheorem{theorem}{Theorem}[section]

\theoremstyle{definition}

\newtheorem{remark}{Remark}
\newtheorem{example}{Example}

\numberwithin{equation}{section}

\makeatletter
\@namedef{subjclassname@2020}{\textup{2020} Mathematics Subject Classification}
\makeatother

\title[Differentiability of limit shapes in continuous FPP]{Differentiability of limit shapes in continuous first passage percolation models}

\author{Yuri Bakhtin}
\address{Courant Institute of Mathematical Sciences, New York University, 251 Mercer St, New York, NY 10012, USA}

\email{bakhtin@cims.nyu.edu}

\author{Douglas Dow}
\address{Courant Institute of Mathematical Sciences, New York University, 251 Mercer St, New York, NY 10012, USA}
\email{dd3103@cims.nyu.edu}


\begin{document}

\begin{abstract}
    We show that for a broad class of continuous first passage percolation models, the boundaries of the associated limit shapes are differentiable.
\end{abstract}
\begin{abstract}
    We introduce and study a class of abstract continuous action minimization problems that generalize continuous first and last passage percolation. In this class of models a limit shape exists. Our main result provides a framework under which that limit shape can be shown to be differentiable. We then describe examples of continuous first passage percolation models that fit into this framework. The first example is of a family of Riemannian first passage percolation models and the second is a discrete time model based on Poissonian points.
\end{abstract}
\maketitle

\section{Introduction}\label{sec:intro}

The main goal of this paper is to show that for a broad class of models of first passage percolation and last passage percolation type in continuous space, the associated shape functions are differentiable away from zero, implying that the boundary of the limit shape is differentiable. 

Optimal paths in disordered environments have been extensively studied in the literature. A variety of interesting models has been introduced. The general scheme is the following: each admissible path $\gamma$ in a Euclidean space $\R^d$ is assigned a random action/cost/energy $A_\omega(\gamma)$ defined through the intrinsic geometry of the path and interactions of the path with the realization of a random environment associated with a random outcome $\omega\in\Omega$. For every pair of points $x$ and $y$ in $\R^d$,
$\Amin_\omega(x,y)$ is defined as the optimal action over the space $\SP_{x,y}$ of admissible  paths connecting $x$ to $y$:
\[
    \Amin_\omega(x,y)=\inf_{\gamma\in\SP_{x,y}} A_\omega(\gamma).
\]
For many interesting models of this kind, one can use stationarity and ergodicity of the environment to apply the subadditive ergodic theorem and prove that the asymptotic growth of action $\Amin_\omega(0,x)$ is linear in the Euclidean norm $|x|$ as $|x|\to\infty$, with rate of growth depending on the direction. More precisely, the limit
\begin{equation}
\label{eq:shape-f}
    \Lambda(v)=\lim_{T\to+\infty} \frac{1}{T}A_\omega(0,Tv)
\end{equation}
is well-defined and deterministic for each $v$ in $\R^d$ 
or, for models of LPP (Last Passage Percolation) type, in a smaller convex cone $\Cone\subset\R^d$ of admissible asymptotic  directions determined by the structure of the set of admissible paths.
For example, in $1+1$-dimensional models of LPP type, where a certain directionality condition 
is imposed on paths (i.e., they cannot backtrack),
$\Cone$ may be the quadrant $\{x\in\R^2: x_1,x_2\ge 0\}$ or the half-plane 
$\{(t,x)\in\R^2: t>0\}$. In FPP (First Passage Percolation) type  models, $\Cone=\R^d$, i.e., there are no restrictions on asymptotic directions of paths.

In some models,  paths and their endpoints are restricted to certain subsets of~$\R^d$ (such as $\Z^d$ or  
$\Z\times\R^{d-1}$),
and since $Tv$ may fail to belong to this set, the claim~\eqref{eq:shape-f} needs to be modified appropriately.

The function $\Lambda:\Cone\to\R$ characterizing the rate of growth of optimal action as a function of direction $v$ is called the shape function, and in the context of homogenization for stochastic Hamilton--Jacobi--Bellman (HJB) equations it can be interpreted as the effective Lagrangian.
The term {\it shape function} comes from the fact that for models where the action $A_\omega(\gamma)$ is nonnegative and plays the role of random length of $\gamma$, the shape function (also nonnegative in this case) can be used to describe the limit shape of normalized balls with respect to the random metric  given by $\Amin_\omega(x,y)$. Namely, for many models one can prove that if 
\begin{equation}
    \label{eq:random_ball}
    \rball_\omega(T)=\{x\in \Cone : \Amin_\omega (0,x)\le T\},
\end{equation}
then, for a properly understood notion of convergence of sets, with probability~$1$, 
\begin{equation}
    \label{eq:limit-sh}
    \lim_{T\to+\infty}\frac{1}{T} \rball_\omega(T)= \rball_\Lambda, 
\end{equation}
where 
\begin{equation}
    \label{eq:limit-shape-via-shapef}
    \rball_\Lambda=\{v\in\Cone: \Lambda(v)\le 1\}.    
\end{equation}
 Thus the set $\rball_\Lambda$
plays the role of the limit shape. Its boundary $\{v\in\Cone: \Lambda(v)=1\}$ plays the role of the effective front characterizing the homogenized wave propagation in the disordered environment. The classical works on limit shapes and shape functions are
\cite{Hammersley1965}, \cite{Kingman:MR0254907}, \cite{Kingman:MR0356192}, \cite{Richardson:MR0329079}, \cite{Cox-Durrett:MR0624685}. Also, see the monograph~\cite{AHDbook:MR3729447} and references therein.

Shape functions are always $1$-homogeneous, i.e., they satisfy $\Lambda(cv)=c\Lambda(v)$ for $c>0$. Also,
due to a simple subadditivity argument, they are always convex. 
Thus, the limit shape $\rball_\Lambda$ is always a convex set. 

Convex functions and boundaries of convex sets may have corners and flat pieces, and  
the problem of characterizing further regularity properties of limit shapes and shape functions beyond simple convexity has been one of the recurrent themes in the theory of FPP and LPP models and stochastic HJB equations.

Typical fluctuations of long minimizers and their actions are tightly related to the regularity of the shape function. 
It is broadly believed that for a vast class of models with fast decay of correlations the shape function and the boundary of the limit shape must be differentiable and strictly convex.
This kind of quadratic behavior of the shape function is associated with the KPZ universality.  Ergodic properties of stochastic HJB equations also depend on the shape function regularity. We refer to \cite{Bakhtin-Khanin-non:MR3816628} for a discussion of this circle of questions.

Despite the importance of the issue, the progress on the regularity properties has been limited. 
It is known since \cite{Haggstrom-Meester:MR1379157} that if one does not require sufficiently fast decay of correlations in the environment, 
any convex set respecting the symmetries of the model can be the limit shape. Thus, any convex 1-homogeneous function with the same symmetries can be realized as shape function. A set of examples with flat edges is provided by
lattice LPP/FPP models with i.i.d.\ environments based on distributions with atoms,  see \cite{Durrett-Liggett:MR0606981}{}. 
In a recent paper \cite{bakhtin2023passage}, an LPP model in nonatomic product-type environment is shown to have both, a corner and a flat edge. We also note that in the deterministic weak KAM theory for spacetime-periodic problems, the shape functions (known as Mather's beta-functions) are strictly convex in the slope variable, differentiable at all irrational slopes, and typically have corners at all rational slopes, see~\cite{Mather:MR1139556}. 

There are several continuous space models with distributional symmetries that translate into a precise analytic form of the shape function, with regularity properties trivially implied.  
For the Hammersley process (and its generalizations), an LPP-type model based on upright paths collecting Poissonian points
from the positive quadrant, the precise form of the shape function is inherited from the fact that linear area-preserving automorphisms of the quadrant preserve the admissible paths and the distribution of the Poisson point process, see~\cite{Hammersley:MR0405665}, \cite{Aldous-Diaconis:MR1355056}, \cite{CaPi}.  
Rotationally invariant Euclidean FPP models introduced in \cite{HoNe3} obviously produce
Euclidean balls as the limit shapes. The LPP type models (including a positive temperature Gibbs polymer version) studied in \cite{BCK:MR3110798}, \cite{kickb:bakhtin2016}, 
\cite{Bakhtin-Li:MR3911894},
\cite{Bakhtin-Li:MR3856947} in the context of the stochastic Burgers equation, allow for a form of invariance under shear transformations resulting in quadratic shape functions. In addition, shape functions have been computed for a handful of exactly solvable 
models, see \cite{Rost:MR635270}, 
\cite{Baryshnikov:MR1818248},\cite{Gravner-Tracy-Widom:MR1830441}, \cite{Hambly-Martin-O'Connell:MR1935124}, \cite{Moriarty-O'Connell:MR2343849}, \cite{Seppalainen:MR2917766},\cite{JRAS:https://doi.org/10.48550/arxiv.2211.06779}. 

It is natural to conjecture that these results can be extended to a broader family of models, on discrete lattices and in continuous spaces. However, they are based on very precise restrictive properties of the models in question, and there seems to be a gap between the universality claims and the concreteness of these models.  In \cite{BakhtinDow_Differentiability} and \cite{bakhtin2023differentiability}, we gave results on differentiability 
of shape functions in the interior of $\Cone$ for a large class of LPP-type models in continuous space. In \cite{BakhtinDow_Differentiability}, we gave a simple argument for $1+1$-dimensional  time-discrete and white-in-time  models, and  in \cite{bakhtin2023differentiability} we adapted our method to multidimensional
spacetime-continuous nonwhite environments. The latter setup allows for an interpretation in terms of differentiability of the effective Lagrangian in the homogenization problem for HJB equations with random forcing. 
Although these models are not distributionally shear-invariant, our argument is based on a form of approximate distributional invariance of the model under a family of shear transformations.

\smallskip

In the present paper, we show that our approach is applicable to continuous space models of FPP type.

In fact, our main result is applicable to continuous models of both FPP and LPP types and loosely 
can be stated as follows:
\begin{theorem}\label{thm:main-informal} Under a set of mild      
    conditions on the random action $A_\omega$ that we describe in Section~\ref{sec:general-results}, there is a deterministic convex function $\Lambda$ such 
    that~\eqref{eq:shape-f} holds for each $v\in\Cone$. Moreover, convergence in~\eqref{eq:shape-f} is uniform on compact sets. The limit shape theorem
    in~\eqref{eq:limit-sh} holds.
    The shape function $\Lambda$ is differentiable at every nonzero interior point of $\Cone$. 
    For positive actions, the effective front (the boundary of the limit shape) is differentiable at all its points in the interior of $\Cone$.
\end{theorem}
\begin{remark}
    We also prove a formula for $\nabla\Lambda$.
\end{remark}

\begin{remark}
    It is known since~\cite{Szn98} (see also \cite{LaGatta_Wehr2010}) that shape functions for FPP type models have a corner at the origin. For example, for rotationally invariant FPP models, the graph of the shape function is a cone with
    spherical section. So for these models, we can claim differentiability of the shape function only at nonzero points. In our previous results on LPP type models, zero was automatically excluded since it did not belong to the interior of the LPP cone~$\Cone$. In fact,  
    we only considered directions of the form $(1,v)\in\R^{1+d}$ (and, by $1$-homogeneity, their multiples).
\end{remark}  

Our results from \cite{bakhtin2023differentiability} on LPP-type models
fit the framework of the present paper (the time-discrete models of~\cite{BakhtinDow_Differentiability} need more adjustments). The conditions we require for our main results here were checked for these models in that paper. See the discussion in Section~\ref{sec:directed}.

Moreover, in the present paper, we
check that the conditions of our main results are satisfied for two classes of anisotropic FPP-type models.
One of them is a random Riemannian metric model in the spirit of~\cite{LaGatta_Wehr2010}, see 
Section~\ref{sec:exampleI} and another
is a random metric based on broken line paths between Poissonian points inspired 
by~\cite{HoNe3}, see Section~\ref{sec:exampleII}.  Our method should apply to a variety of similar models, but we chose these two classes where the application is relatively straightforward.

\bigskip

The paper is organized as follows. In Section~\ref{sec:general-results} we describe the general setup, state the main general conditions and main results, rigorous counterparts of the informal Theorem~\ref{thm:main-informal}. We also give a proof of our  central differentiability result in its general form in that section. 
Proofs of all the other results stated in this section are postponed until
Section~\ref{sec:uniformConvergence}. 
In Sections~\ref{sec:exampleI} and~\ref{sec:exampleII}, we describe two classes of FPP-type models that our results apply to. In Section~\ref{sec:directed}, we explain that the results from~\cite{bakhtin2023differentiability} fit the framework of the present paper. The remaining sections contain proofs of various results from first four sections. Section~\ref{sec:uniformConvergence} contains proofs of the results from Section~\ref{sec:general-results}.
Sections~\ref{sec:proof-of-main-riemannian}--\ref{sec:appendix} contain proofs of the results from Section~\ref{sec:exampleI}. Section~\ref{sec:proofs_broken-lines} contains the proof of the result from Section~\ref{sec:exampleII}

{\bf Acknowledgments.} YB and DD are grateful to the National Science Foundation for 
partial support via Awards DMS-1811444 and DMS-2243505. We thank Peter Morfe for pointing to  \cite{tu2024regularity}, where a method similar to ours is used in a related but different problem.

\section{General conditions and results}\label{sec:general-results}

In this section, we give a general framework and state our main results. Our focus is on the fully continuous case, though some discrete in time problems can be embedded into our framework. In Section \ref{sec:generalSetUp}, we introduce a general set of assumptions and state the (standard) results on convergence to the shape function and limit shape. The central part of this paper is Section~\ref{sec:differentiability}
where we state our main results on differentiability.

\subsection{General setup and standard limit shape results.}\label{sec:generalSetUp}

First we will define the spaces we will work with. For $x,y\in \R^d$ and $t>0$, let 
\[
\SP_{x,y,t} = \{\gamma\in W^{1,1}([0,t];\R^d)\,:\,\gamma_0=x,\,\gamma_t=y\}.
\]
Then we can define
\[\SP_{x,y,*}=\bigcup_{t>0}\SP_{x,y,t},\quad x,y\in\R^d,\]
\[\SP=\SP_{*,*,*} = \bigcup_{x,y\in\R^d,\, t>0}\SP_{x,y,t},\quad x,y\in\R^d,\ t>0,\]
and other similar spaces such as $\SP_{*,*,t}$.
For $\gamma\in\SP_{*,*,t}$, we define $t(\gamma)=t$.

The space $\SP$ is a separable metric space when equipped with the Sobolev metric given by 
\begin{equation}
    d(\gamma,\psi) = \int_0^1 |\gamma_{t_1s} - \psi_{t_2s}|ds + \int_0^1 |t_1\dot{\gamma}_{t_1s} - t_2\dot{\psi}_{t_2s}|ds + |t_1-t_2|
\end{equation}
for $\gamma\in \SP_{*,*,t_1}$ and $\psi\in\SP_{*,*,t_2}.$
The spaces $\SP_{x,y,t}$ and $\SP_{x,y,*}$ are endowed with the induced topology, which coincides with the $W^{1,1}$ topology on these spaces.

For any $x,y,z\in\R^d$ and $t_1,t_2>0$, if  $\gamma\in\SP_{x,y,t_1}$ and $\psi\in \SP_{y,z,t_2}$, then $\gamma\psi\in \SP_{x,z,t_1+t_2}$ denotes their concatenation, defined by
\[
(\gamma\psi)_s=
\begin{cases}
\gamma_{s},& s\in[0,t_1],\\
\psi_{s-t_1},& s\in[t_1,t_2].
\end{cases}
\]
For $x\in\R^d$, the spatial shift $\theta^x:\R^d\to\R^d$ is defined by
\begin{equation}
    \label{eq:theta}
    \theta^xy=y+x,\quad y\in\R^d. 
\end{equation}
This definition lifts to transformations of $\SP$, namely, for $x\in\R^d$, $t>0$, and $\gamma\in \SP_{*,*,t}$, the spatial shift  $\theta^x\gamma\in \SP_{*,*,t}$ is defined by
\[ (\theta^x \gamma)_s=\gamma_{s}+x,\quad s\in[0,t].\]

We consider a complete probability space $(\Omega,\mathcal{F},\Prb)$ and assume that the group~$\R^d$ acts on $\Omega$ 
ergodically. Namely, we assume that we are given a family $(\theta_*^x)_{x\in\R^d}$ of measurable transformations
 $\theta_*^x:\Omega\to \Omega$ preserving $\Prb$, ergodic and 
having the group property: for all $x,y\in \R^d$, $\theta_*^{x+y} = \theta_*^x\theta_*^y$ and $\theta_*^0$ is the identity map.

We consider a jointly measurable action/energy/cost function 
\begin{align*}
A:\Omega\times\SP &\to \R\cup\{\infty\},\\
 (\omega,\gamma)&\mapsto A_\omega(\gamma),
\end{align*}
which is a random field indexed by absolutely continuous paths. Once $\gamma\in\SP$ is fixed,  $A(\gamma)=A_\cdot(\gamma):\Omega\to \R\cup\{\infty\}$ is a random variable. Once $\omega\in\Omega$ is fixed, $A_\omega:\SP\to\R\cup \{\infty\}$ assigns actions to all absolutely continuous paths. One of the goals of allowing for infinite action values is to accommodate the LPP settings where only certain ``directed'' paths are admissible. Nonadmissible paths will be assigned infinite action and thus will be excluded from variational problems.
 
The following two assumptions on $A$ are fundamental. The first relates the action of $A$ on the concatenation of paths to the sum of the actions of each individual path. In the common set-up where $A$ is given by a local energy summed or integrated along a path the relation is an identity rather than an inequality. The second assumption, called skew-invariance, implies statistical stationarity and ergodicity conditions on $A$.

\begin{enumerate}[label=(A\arabic*), ref = \rm{(A\arabic*)}]
    \item\label{cond:subbadd}(subadditivity) For all $\omega\in\Omega$, $x,y,z\in \R^d$, $\gamma\in \SP_{x,y,*}$, and $\psi \in \SP_{y,z,*}$, 
    \[A_\omega(\gamma\psi)\le A_\omega(\gamma) + A_\omega(\psi).\]
    \item\label{cond:skew-invariance}(skew-invariance) For all $\omega\in \Omega,$ $x\in \R^d$,  all $\gamma\in \SP$, 
    \begin{equation*}
        A_\omega(\gamma) = A_{\theta_*^{x}\omega}(\theta^x\gamma).
    \end{equation*}
\end{enumerate}

We want to study the minimization problem 
\begin{equation}\label{eq:AStarDef}
    \Amin(x,y)=\Amin_\omega(x,y) := \inf\{A(\gamma)\,:\,\gamma\in \SP_{x,y,*}\}.
\end{equation}
We will need the following assumption: 
\begin{enumerate}[resume*]

    \item \label{cond:measurability}
The minimal action $\Amin$ is jointly measurable as a function from $\Omega\times \R^d\times \R^d$ to $\R\cup\{+\infty\}$. 
\end{enumerate}

In order to accommodate LPP-type problems, we need to introduce a cone $\Cone\subset \R^d$ of admissible directions. For the Hammersley process the role of $\Cone$ is played by the positive quadrant. For HJB equations with dynamic random forcing considered in \cite{bakhtin2023differentiability}, the role of $\Cone$
is played by the half-space $(0,\infty)\times\R^{d-1}$. In FPP-type problems with no constraints on directions of paths, $\Cone=\R^d$. We require the following properties of the cone $\Cone$ and the action  $A$:

\begin{enumerate}[resume*]
    \item \label{cond:cone} $\Cone\subset \R^d$ is a convex and nonempty cone. 
    For all $\omega\in\Omega$ and for every $x\in \Cone$, 
    there is a random path $\gamma(x)=\gamma_\omega(x)\in \SP_{0,x,*}$ achieving the infimum in~\eqref{eq:AStarDef}, i.e., $\Amin_\omega(0,x) = A(\gamma_\omega(x))$ for all $\omega\in \Omega$.
    Additionally, for all $x\in \Cone$, $\sup_{r\in [0,1]}|\Amin(0,rx)|$ is measurable and
\begin{equation}
    \label{eq:expected_action_finite}
    \E\Big[\sup_{r\in [0,1]}|\Amin(0,rx)|\Big] < \infty.     
\end{equation}
\end{enumerate}

 Conditions \ref{cond:skew-invariance} and  \ref{cond:cone}  imply that $\E[|\Amin(x,y)|] < \infty$ for all $x,y\in \R^d$ satisfying $y-x \in \Cone.$ For simplicity of presentation we assume that conditions \ref{cond:subbadd}--\ref{cond:cone} hold for all $\omega\in \Omega$, but with minor adjustments one may allow for a single exceptional set of zero measure on which the conditions of \ref{cond:subbadd}--\ref{cond:cone} fail.

For $v\in \mathcal{C}$ define 
\begin{equation}
    \Amin^T(v) = \Amin(0,Tv).
\end{equation}

\begin{theorem}\label{thm:shapeFunction} Under assumptions \ref{cond:subbadd} --- \ref{cond:cone}, 
    there is a convex, deterministic function $\Lambda:\mathcal{C}\to [-\infty,\infty)$ such that for all $v\in \mathcal{C}$, with probability one,
    \begin{equation}
        \label{eq:conv_to_Lambda}
        \Lambda(v) = \lim_{T\to \infty}\frac{1}{T}\Amin^T(v).
    \end{equation}
    Additionally, $\Lambda(sv) = s\Lambda(v)$ for all $v\in \R^d$ and $s > 0$,
    and if $0\in\Cone$, then $\Lambda(0)=0$.
\end{theorem}
Kingman's Subadditive Ergodic Theorem was proved with problems like this in mind. We remind the standard argument 
in Section~\ref{sec:uniformConvergence} for completeness.

\smallskip

In all of our examples, $\Lambda$ is, in fact, finite. 
Let us supplement our setup with an additional assumption guaranteeing finiteness of $\Lambda$ and uniform convergence in~\eqref{eq:conv_to_Lambda}. We need the latter to prove convergence to a limit shape. 

 We say that a cone $\Cone' \subset \Cone$ is \textit{properly contained} in $\Cone$ if $\overline{\Cone'}\setminus\{0\}\subset \Cone^{\circ}$.

 \bigskip

\begin{enumerate}[resume*]
    \item \label{cond:for-finiteness}
For every cone $\Cone'\subset \Cone$ properly contained in $\Cone$, there is $\kappa < \infty$ such that 
\begin{equation}\label{eq:linearGrowthProb}
    \Prb\bigg\{ \sup_{x\in \Cone',\,|x|>1}\frac{|\Amin(0,x)|}{|x|} < \kappa\bigg\} > 0,
\end{equation}
and
\begin{equation}\label{eq:linearGrowthProb2}
    \Prb\bigg\{ \sup_{x\in \Cone',\,|x|>1}\frac{|\Amin(-x,0)|}{|x|} < \kappa\bigg\} > 0.
\end{equation}
\end{enumerate}
In most applications, \eqref{eq:linearGrowthProb} and \eqref{eq:linearGrowthProb2} are equivalent due to distributional symmetries of the action. Although in some settings \eqref{eq:linearGrowthProb} and \eqref{eq:linearGrowthProb2} hold for $\Cone' = \Cone$, there are LPP settings where the action $\Amin(0,x)$ goes to infinity as $x$ approaches the boundary of $\Cone$, and so using the notion of properly contained cones is unavoidable.

\begin{theorem}\label{thm:uniformConvergenceShape}
    Under assumptions \ref{cond:subbadd} --- \ref{cond:for-finiteness}, $\Lambda(v) > -\infty$
    for all $v\in\Cone^{\circ}$ and there is a full measure set $\Omega_0$ such that for all $\omega\in \Omega_0$ and all compact sets $K\subset \Cone^{\circ},$
    \begin{equation}\label{eq:uniformConvergence}
        \lim_{T\to \infty}\sup_{w\in K}\Big|\frac{1}{T}\Amin^T(w) - \Lambda(w)\Big|  = 0.
    \end{equation}
\end{theorem}
We will prove this theorem in \Cref{sec:uniformConvergence}.

These results can be understood in terms of limit shapes. Limit shapes are usually 
defined in the context of FPP, where $\Cone=\R^d$, the action $\Amin(x,y)$ is positive and
 can be interpreted as a random metric between points $x$ and $y$.
The set $\rball_\omega(T)$ defined in~\eqref{eq:random_ball} can be viewed as a ball of radius $T$ in this metric.  For any set $K\subset\R^d$ and a number $a\in\R$,
we denote $aK=\{ax: x\in K \}$. 

We will say that a family of sets $N_T\subset \Cone,$ $T>0$,
converges locally to a set $N\subset \Cone$ and write
\[ N_T\stackrel{\mathrm{loc}}{\longrightarrow} N,\quad T\to\infty,\]
if for every compact set $K\subset \Cone^{\circ}$ and every $\epsilon>0$, there is $T_0>0$ such that
\begin{equation}
    \label{eq:local-convergence-of-shapes}
    ((1-\epsilon)N) \cap K \subset N_T \cap K \subset  ((1+\epsilon) N) \cap K,\quad T>T_0.
\end{equation}

The following result shows that $\rball_\Lambda$ defined in~\eqref{eq:limit-shape-via-shapef} is the deterministic limit shape associated with the
random action $A_\omega$. 
\begin{theorem} 
    \label{thm:limit-shape}
    Under assumptions \ref{cond:subbadd} --- \ref{cond:for-finiteness},
    with probability $1$, 
    \begin{equation}\label{eq:shapeConvergenceCompact}
        \frac{1}{T} \rball_\omega(T)\stackrel{\mathrm{loc}}{\longrightarrow}\rball_\Lambda, \quad T\to\infty.  
    \end{equation}
\end{theorem}

\begin{remark}\label{rem:limit-shape-strong}
     In fact, a stronger form of convergence often holds. Namely, if in addition to the conditions of 
    Theorem~\ref{thm:limit-shape}, we require 
    $\Cone=\R^d$ and $\Lambda(v) > 0$ for all $v\neq 0$ (this holds for  a typical FPP setting and, in particular, for our examples studied in Sections~\ref{sec:exampleI} and \ref{sec:exampleII}), we can prove that with probability 1,
    for every $\epsilon>0$, there is $T_0>0$ such that  
    \begin{equation}
        \label{eq:shapeConvergence-strong}
        (1-\epsilon)\rball_\Lambda \subset  \frac{1}{T} \rball_\omega(T)\subset  (1+\epsilon) \rball_\Lambda ,\quad T>T_0,
    \end{equation}
    i.e., a version of \eqref{eq:local-convergence-of-shapes} with $K$ replaced by $\Cone=\R^d$ holds. We prove this claim along with 
    Theorem~\ref{thm:limit-shape} in Section~\ref{sec:proof-limit-shape}.
\end{remark}
\begin{remark}
   In directed settings, where the cone boundary $\partial\Cone$ is nonempty, the stronger convergence \eqref{eq:shapeConvergence-strong} is also often true, and can be derived from a stronger version of
   Theorem~\ref{thm:uniformConvergenceShape}, where the uniform convergence holds up to $\partial\Cone$. 
   In our general setting, even continuity of $\Lambda$ up to the boundary is not guaranteed, and one needs extra regularity conditions to ensure nice behavior of $\Amin$ and $\Lambda$ near $\partial \Cone$. Paths on the boundary of $\Cone$ are more constrained than paths in $\Cone^{\circ}$ and so in practice different techniques are often used near the boundary (see for e.g. \cite{MR2094434}). We do not address these issues in detail and concentrate on differentiability in the interior of~$\Cone$.

\end{remark}

\subsection{Differentiability}\label{sec:differentiability}
In this section, we assume the setup described in Section~\ref{sec:generalSetUp} and present general conditions guaranteeing differentiability of $\Lambda$ at a point $v\in\Cone^\circ\setminus\{0\}$. Our main result is Theorem~\ref{thm:mainThm}.
We will see below (Lemma~\ref{lem:diffLemma}) that the desired differentiability of $\Lambda$ at $v$ follows from differentiability along a sufficiently large set of directions. Thus the assumptions introduced in this section will be targeted at checking this directional differentiability. These assumptions are easy to verify. In 
Sections~\ref{sec:exampleI} and \ref{sec:exampleII} we will check them for two classes of FPP-type models. The proofs of differentiability for LPP-type models in \cite{BakhtinDow_Differentiability}, 
\cite{bakhtin2023differentiability} were essentially based on
checking these assumptions, too.

We denote
$\Ball(x,r)=\{y\in \R^d:\ |x-y|<r\}$. We need the following in our setup:

\begin{enumerate}[label=(B\arabic*), ref=\rm{(B\arabic*)}]
    \item  \label{cond:setup_diff}There is $\delta > 0$ and
    a $d-1$ dimensional subspace 
    $H\subset \R^d$ 
 not containing~$v$ with the following properties: for 
 all $T>0$ and $w\in  H(\delta)$, where
\begin{equation}
    \label{eq:H-delta}
H(\delta)= (v + H)\cap \Cone^\circ\cap \Ball(v,\delta),
\end{equation} 
 there is a pair of maps:
a measurable bijection $\Xi_{v\to w}:\SP_{0,Tv,*}\to \SP_{0,Tw,*}$ 
and a measure preserving map $\Xi_{v\to w}^*:\Omega\to \Omega.$  In addition, $\Xi_{v\to v}^*$  and $\Xi_{v\to v}$ are identity maps.
\end{enumerate}

We drop the dependence on $T$ from $\Xi_{v\to w}$ and  $\Xi_{v\to w}^*$ for brevity.
In applications, the maps $\Xi_{v\to w}$ are usually lifted from 
a transformation of $\R^d$ that does not depend on $T$.

Before describing the further requirements on $\delta,H$ and maps $\Xi_{v\to w}$, $\Xi^*_{v\to w}$, we need to introduce 
further notation.

The function $B$ defined as 
\begin{equation}\label{eq:AB_equality}
    B_\omega(w,v,\gamma) = A_{\Xi_{v\to w}^* \omega}(\Xi_{v\to w}\gamma),\quad \gamma\in\SP_{0,Tv,*},
\end{equation}
is a transformed version of $A$.

The optimal transformed action is given by
\begin{equation}\label{eq:Bdef}
    \Bmin^T(w,v) = \inf\{B(w,v,\gamma)\,:\,\gamma\in \SP_{0,Tv,*}\}.
\end{equation}
If $\gamma^T(v)=\gamma^T_{\omega}(v)=\gamma_\omega(Tv)$ is the selection of path realizing the infimum in the definition of $A^T(Tv)$, 
see condition \ref{cond:cone}, we define 
\[\psi^T(w,v) = \psi^T_\omega(w,v) = \Xi_{v\to w}^{-1}\gamma^T_{\Xi_{v\to w}^*\omega}(w) \in \SP_{0,Tv,*}\] 
to be our selection of path obtaining the infimum in \eqref{eq:Bdef}. 

Since $\Xi_{v\to v}$ and $\Xi_{v\to v}^*$ are identity maps, we have
\begin{equation}
\label{eq:psi_is_gamma}
\psi^T(v,v)=\gamma^T(v)=\gamma(Tv),
\end{equation}
and
\begin{equation}
 \label{eq:B-minimizer}
 \Bmin^T(v,v) = B(v,v,\gamma^T(v)).
\end{equation}

By \eqref{eq:AB_equality} and the assumption that $\Xi_{v\to w}^*$ is measure preserving, for every $w\in (v + H)\cap\Cone$, 
\begin{equation}\label{eq:B_ShapeTheorem}
    \Lambda(w) = \lim_{T\to \infty}\frac{1}{T}\Bmin^T(w,v)
\end{equation}
$\Prb$-almost surely. 

For a function $f:H(\delta)\to \R$ ($H(\delta)$ is defined in~\eqref{eq:H-delta}), one can define $\nabla_H f$ and $\nabla_H^2 f$ as, respectively, its first and second derivatives relatively to~$H$.
These derivatives can be identified as elements of $H$ and $H^2$, respectively, so that
\[
f(w') = f(w) + \langle\nabla_H f(w),w'-w \rangle +  \langle \nabla^2_H f (w)(w'-w), w'-w\rangle +o(|w'-w|^2), 
\]
as $H(\delta)\ni w'\to w$, 
where the inner product in $H$ is induced by the inner product 
in~$\R^d$. Note that if $f$ is in fact defined on an open set in $\R^d$ and is twice differentiable at $v,$ then $\nabla_H f = P_H \nabla f$ and $\nabla_H^2 f = P_H\nabla^2 f$, where $\nabla f$ and $\nabla^2 f$ are the usual derivatives in $\R^d$ and $P_H$ is the orthogonal linear projection onto $H.$

Since $v$ is fixed, we consider $B(w,v,\gamma^T(v))$ as a function of its first argument $w$.

\bigskip

We are now ready to state the crucial assumption and the main differentiability theorem:

\begin{enumerate}[resume*]
    \item \label{cond:M_infty}
 $\nabla_H B(w,v,\gamma^T(v))$ and $\nabla_H^2 B(w,v,\gamma^T(v))$ exist for all $w\in H(\delta)$, and
\begin{equation}
    \label{eq:M_infty}
    M_\infty := \limsup_{T\to\infty}\frac{1}{T}\sup_{w\in H(\delta)} \|\nabla_H^2 B(w,v,\gamma^T(v))\|  < \infty
\end{equation}
almost surely.
\end{enumerate}

\begin{theorem}\label{thm:mainThm} Under assumptions \ref{cond:subbadd} --- \ref{cond:cone} and \ref{cond:setup_diff}---\ref{cond:M_infty}, the function $\Lambda$ is differentiable at $v.$ Additionally, for $w\in H,$
    \begin{equation}
        \label{eq:H-diff-of Lambda}
        \langle  \nabla \Lambda(v), w\rangle = \lim_{T\to \infty} \frac{1}{T}\langle \nabla_H B(v,v,\gamma^T(v)),w\rangle .
    \end{equation}
\end{theorem}
\begin{remark} A convex function is differentiable on an open set iff it is $C^1$ on that set. Thus, if the conditions of the theorem hold for all $v\in\Cone^\circ$, then  $\Lambda\in C^1(\Cone^\circ)$. 
\end{remark}

The proof is based on the representation~\eqref{eq:B_ShapeTheorem} and the following two lemmas. The first of them implies that it suffices to check differentiability of $\Lambda$ relatively to~$H$. The second one  is at the core of our argument. It is a minor modification of Lemma 3.3 of \cite{bakhtin2023differentiability}. We give a proof of these lemmas in Section~\ref{sec:uniformConvergence}.
\begin{lemma}\label{lem:diffLemma}
    Let $f:\Cone^\circ\to \R$ be a continuous function such that its restriction to $H(\delta)$ is differentiable at $v$.
    Also, suppose $f$ satisfies $f(sw) = s f(w)$ for $w$ in a neighborhood of $v$ and all sufficiently small $s>0$. Then $f$ is differentiable at $v$.
\end{lemma}

\begin{lemma}\label{lem:deterministicDifferentiability}
    Let $\mathcal{D}\subset H(\delta)$ be dense in $H(\delta)$, $(f_n)_{n\in \N}$ be a sequence of functions from $H(\delta)$ to $\R$, and $f:H(\delta)\to \R$ be a function such that for all $w\in \mathcal{D}\cup\{v\}$,
    \begin{equation}
        \label{eq:f_n-to-f}
        \lim_{n\to \infty}f_n(w) = f(w).
    \end{equation}
    Suppose also that there exists a sequence of vectors $(\xi_n)_{n\in \N}$ in $H$ and a function $h:H(\delta)\to \R$ such that the following holds:
    \begin{enumerate}
        \item\label{linearDomination1} For all $w\in \mathcal{D}$ and $n\in \N,$ 
        \begin{equation}\label{finiteDerivativeInequality}
            f_n(w) - f_n(v) \le \langle \xi_n, w-v \rangle + h(w),
        \end{equation}
        \item\label{linearDomination2} $\lim_{w\to v}\frac{h(w)}{|w-v| } = 0$.
    \end{enumerate}
    If $f$ is convex, then: $f$ is differentiable at $v$ (relatively to $H$), the sequence $(\xi_n)_{n\in \N}$ converges, and 
    \begin{equation}\label{convexImpliesDifferentiable}
        \nabla_H f(v) = \lim_{n\to \infty}\xi_n.
    \end{equation}
\end{lemma}

\begin{proof}[Proof of Theorem~\ref{thm:mainThm}]

    Let $w\in v+H$ be such that $|w-v|<\delta.$ We have 
    \begin{align*}
        \Bmin^T(w,v) &\le B(w,v,\gamma^T(v))\\
        & = B(v,v,\gamma^T(v)) +  B^T(w,v,\gamma^T(v)) - B(v,v,\gamma^T(v))\\
        & \le \Bmin^T(v,v) + \langle \nabla_H B(v,v,\gamma^T(v)),w-v\rangle 
        \\ &\qquad\qquad\qquad\qquad\qquad + \frac{1}{2}\sup_{u\in H(\delta)}\|\nabla^2_H B^T(u,v,\gamma^T(v))\|\cdot|v-w|^2 ,
    \end{align*}
    where we used~\eqref{eq:B-minimizer} and the Taylor expansion.
    It follows that for sufficiently large~$T$,
    \begin{multline}\label{eq:linearDominationB}
        \frac{1}{T}\Bmin^T(w,v) \le \frac{1}{T}\Bmin^T(v,v) + \frac{1}{T}\langle \nabla_H B^T(v,v,\gamma^T(v)),w-v\rangle \\ + \frac{1}{2}|w-v|^2 (M_\infty + 1),
    \end{multline}
    where $M_\infty$ is as defined in \eqref{eq:M_infty}. 
Taking an arbitrary countable dense set $\mathcal{D}\subset H(\delta)$, $f_n(\cdot)=n^{-1}\Bmin^{n}(\cdot,v)$,  $f=\Lambda$,  $\xi_n=\nabla_H B(v,v,\gamma^T(v)) $, $h(w)=\frac{1}{2}|w-v|^2(M_\infty+1)$, and noticing that~\eqref{eq:f_n-to-f} for $v\in\mathcal{D}\cup\{v\}$
is a consequence of~\eqref{eq:B_ShapeTheorem},  \eqref{finiteDerivativeInequality} is a consequence of \eqref{eq:linearDominationB}, and recalling that $\Lambda$ is convex, we apply  Lemma~\ref{lem:deterministicDifferentiability}  to 
conclude that $\Lambda$ is differentiable at $v$ relative to $H$ with derivative given by~\eqref{eq:H-diff-of Lambda}. To complete  the proof of the lemma, it now suffices to apply Lemma~\ref{lem:deterministicDifferentiability}.
\end{proof}

\begin{remark}
    If instead of \eqref{eq:M_infty} we assume that for some $v,w\in \Cone^{\circ}$ there is $\delta > 0$ such that 
    \begin{equation}\label{eq:M_infty2}
        \limsup_{T\to\infty}\frac{1}{T}\sup_{w'\in w+H\,:\,|w'-w|<\delta} \|\nabla_H^2 B(w',v,\psi^T(w,v))\|  < \infty,
    \end{equation}
    then a similar proof to that of \Cref{thm:mainThm}, except using a Taylor expansion around $w$ rather than around $v$, shows that $\Lambda$ is differentiable at $w$ and for all $u\in H,$
    \[ \langle \nabla \Lambda(w),u\rangle = \lim_{T\to \infty} \frac{1}{T}\langle\nabla_H B(w,v,\psi^T(w,v)), u\rangle\]
    almost surely.
\end{remark}

\begin{remark}\label{rem:orthogonalDirectionsRemark}
    One can also state a version of \Cref{thm:mainThm} where $H(\delta)$ is replaced by a  $(d-1)$-dimensional $C^2$-hypersurface in $\R^d$ containing $v$ and transversal to the radial direction at $v$, provided that the analogous condition to \eqref{eq:M_infty} holds when $\nabla_H^2 B$ is replaced by the intrinsic second derivative. 
\end{remark}

\begin{remark}\label{rem:HolderRemark}
    The $C^2$ condition in \eqref{eq:M_infty} or \eqref{eq:M_infty2} could be replaced by an analogous $C^{1 + \alpha}$ condition (that is, a bound on the H\"{o}lder constant of the first derivative). In this case, a bound similar to \eqref{eq:linearDominationB} will hold except the $|v-w|^2$ term is replaced by $|v-w|^{1+\alpha}.$ Since this term is $o(|v-w|)$, \Cref{lem:deterministicDifferentiability} can still be applied.
\end{remark}

Finally, 
we can state a result on differentiability of the boundary of the limit shape defined by 
    \[M=\partial \rball_\Lambda \cap \Cone^\circ =\{v\in \Cone^\circ: \Lambda(v)=1 \}.\] 
The following theorem follows directly from
\Cref{thm:mainThm}, \Cref{rem:limit-shape-strong}, $1$-homogeneity of $\Lambda$, and the implicit function theorem.
\begin{theorem}\label{thm:limit-shape-diff} 
        Suppose $M$ is nonempty, and assume that the conditions of Theorem~\ref{thm:mainThm} are satisfied for all 
        $v\in M$. Then $M$ is  a $C^1$  manifold. If
        \begin{align}
            \label{cond:Lambda-cone}
            \Cone=\R^d \text{\ and }\  \Lambda(v)>0 \text{\ for all } v\ne 0,
        \end{align}
        then $M$ is $C^1$-diffeomorphic to the $(d-1)$-dimensional sphere.
\end{theorem}
Condition \eqref{cond:Lambda-cone} means that the graph of $\Lambda$  is a  cone with vertex at the origin and  section $M$.

Theorem~\ref{thm:main-informal} stated informally in Section~\ref{sec:intro} is a combination of rigorous
Theorems~\ref{thm:shapeFunction}--\ref{thm:limit-shape-diff}. 

Of course, the power of these theorems is that they apply to a broad class of situations satisfying requirements~\ref{cond:subbadd} -- \ref{cond:for-finiteness} and \ref{cond:setup_diff} -- \ref{cond:M_infty}.
In Sections~\ref{sec:exampleI} and~\ref{sec:exampleII}, we give two families of such models of FPP type. In Section~\ref{sec:directed}, we explain that these results also apply to the directed 
setting of stochastic HJB equations studied in~\cite{bakhtin2023differentiability}.

\section{Example I: Riemannian First Passage Percolation}\label{sec:exampleI}

The goal of this section is to describe a class of situations 
where the general requirements of Section~\ref{sec:general-results} hold.  This class is
defined via random Riemannian metrics.

\subsection{General setup for Riemannian FPP}\label{sec:RiemannianFPPeg}
  A Riemanninan metric on $\R^d$ is a function from $\R^d$ to the space of positive definite symmetric matrices 
\[\mathcal{M}^d_+ := \{M\in \R^{d\times d}\,:\, M^\top=M,\, M \text{\rm\ is positive definite\}}.\]
Note that $M\in \mathcal{M}^d_+$ can be identified with a quadratic form given by $M(u,u)=\langle Mu, u\rangle.$

For an absolutely continuous curve $\gamma:[0,t]\to\R^d$, its length under a Riemannian metric~$g$ defined by
 \begin{equation}\label{eq:AgammaDefFPP1}
     A(\gamma) = \int_0^t \sqrt{g_{\gamma_s}(\dot{\gamma}_s,\dot{\gamma}_s)}ds
 \end{equation}
plays the role of action. The distance between arbitrary $x,y\in\R^d$ is defined by 
 \begin{equation}\label{eq:ADef}
     \Amin(x,y) = \inf_{\gamma\in \SP_{x,y,*}}A(\gamma),
 \end{equation}
 and throughout this section the cone of admissible direction $\Cone$ is set to be $\R^d$.

 We will use the notation $\|M\|$ for the operator norm of a matrix~$M$.
For a $C^2$ function $f:\R^d\to \mathcal{M}^d_+$, we let 
\[\|f\|_{C^2,x} = \|f(x)\| + \sum_{i=1}^d \|\partial_{x_i}f(x)\| + \sum_{i,j=1}^d\|\partial_{x_jx_i}f(x)\|,\]
\[\|f\|_{C^2} = \sup_{x\in \R^d} \|f\|_{C^2,x}.\]
The space $C^2_{loc}(\R^d;\mathcal{M}^d_+)$ is the space of such functions $f$ such that $\sup_{x\in A}\|f\|_{C^2,x} < \infty$ for all bounded sets $A\subset \R^d$. We endow $C^2_{loc}(\R^d;\mathcal{M}^d_+)$ with the topology induced by the family of semi-norms given by $\sup_{x\in \Ball(0,n)}\|f\|_{C^2,x}$ for $n\in \N$. 

In this section, we require that $g$ is a random element of $C^2_{loc}(\R^d;\mathcal{M}^d_+)$, i.e., it is a measurable map
$g:\Omega\to C^2_{loc}(\R^d;\mathcal{M}^d_+)$. For a fixed $x\in \R^d$, this map gives a random matrix $g_x=g_{x,\omega}$.

\medskip

Let us  specify a set of general conditions that guarantee that the setting and assumptions of Section~\ref{sec:general-results} hold for
the distance given in~\eqref{eq:AgammaDefFPP1}--\eqref{eq:ADef}. Then our results on the shape function and limit shape including the differentiability result will be applicable to this class of models. 

\smallskip

Our first condition concerning stationarity and skew-invariance of $g$ is stated in terms of spatial translations defined in~\eqref{eq:theta}:
\begin{enumerate}[label=(C\arabic*), ref=\rm{(C\arabic*)}]
     \item\label{stationaryPhiCondition}
     The probability space $(\Omega,\Fc,\Prb)$ is equipped with an ergodic $\Prb$-preserving group action  $(\theta_*^x)_{x\in\R^d}$ 
synchronized with translations $(\theta^x)_{x\in\R^d}$ on $\R^d$: 
     for every $x\in \R^d$, $\omega\in \Omega,$ and $y\in\R^d$, we have $g_{\theta^y x,\theta^y_\ast \omega} = g_{x,\omega}.$
\end{enumerate}

\smallskip
For our second condition, we need to fix an arbitrary $v\ne 0$, introduce a space~$H$, and
define a family of transformations $(\Xi_{v\to w})_{w\in v+ H}$ that will satisfy the conditions in \ref{cond:setup_diff}. 
Once $v$ is fixed, we define $H$ as the orthogonal complement to the line spanned by $v$. For $v,w\neq 0$ we define the transformation $\Xi_{v\to w}$ of $\R^d$ by
\begin{equation}\label{eq:XiFPP}
    \Xi_{v\to w} x = \frac{\langle v,x\rangle}{|v|^2} w  - \frac{\langle v,x\rangle  }{|v|^2}v + x=
    \frac{\langle v,x\rangle}{|v|^2} (w-v) + x.
\end{equation}
This is a convenient choice for the models we are mostly concerned with but  
other choices of $H$ and $\Xi_{v\to w}$ are possible. 

The map $\Xi_{v\to w}$ acts on $\SP$ in a pointwise manner: $(\Xi_{v\to w}\gamma)_s = \Xi_{v\to w}\gamma_s.$ Note that $\Xi_{v\to v}$ is the identity map. These maps satisfy \ref{cond:setup_diff} and preserve volume, which  is useful in applications. We summarize these facts in the lemma below.

\begin{lemma}\label{lem:XiProperties}
    If $v\ne 0$ and $w\in v+H$, then $\Xi_{v\to w}$ is a volume preserving transformation on $\R^d$; additionally, $\Xi_{v\to w}$ is a measurable bijection from $\SP_{0,Tv,*}$ to $\SP_{0,Tw,*}$.
\end{lemma}
We postpone a proof of these properties until~\Cref{sec:appendix}. Here we only mention that the volume preserving property implies that
the homogeneous Poisson process in $\R^d$ (that our example models will be based upon) is distributionally invariant under $\Xi_{v\to w}$ for all $w\in v+H$.

Our next condition means that  $g$ and its image under $\Xi_{v\to w}$ can be efficiently
coupled, with controlled errors.
 \begin{enumerate}[resume*]
    \item\label{phiwDiffCondition}
     There is a family of $\Prb$-preserving transformations $(\Xi_{v\to w}^*)_{w\in v+H}$ on $\Omega$ (with $\Xi_{v\to v}^*$ the identity), a number $\delta\in(0,1)$ and a random field $\YField:\R^d\to [0,\infty)$ 
      such that:
      \begin{enumerate}[label = (\roman*), ref=\theenumi{}\rm{\roman*}]
     \item\label{gBound} for all $x\in\R^d$, all $w\in \R^d$ satisfying $|w-v|<\delta$ and all $\omega\in \Omega,$ 
     the random field $g^{w,v}_{x,\omega} := g_{\Xi_{v\to w} x, \Xi_{v\to w}^* \omega}$ satisfies 
     \begin{equation}\label{eq:boundForg}
        \|g_x^{w,v}\| + \sum_{i=1}^d\|\partial_{w_i}g_x^{w,v}\| + \sum_{i,j=1}^d\|\partial_{w_j}\partial_{w_i}g_x^{w,v}\| \le \YField(x);
     \end{equation}
    \item\label{boundedY} $\YField$ satisfies the following conditions:
    \begin{enumerate}[label = (\alph*)]
        \item\label{Ystationary}(stationarity) $\YField$ is stationary with respect to lattice shifts: For all $a\in \Z^d,$ the collection $(\YField(x+a))_{x\in \R^d}$ is equal in distribution to $(\YField(x))_{x\in \R^d}$;
        \item\label{YFiniteRange}(finite range) $\YField$ has a finite range of dependence: there is $R>0$ such that if  $\inf_{x\in A,y\in B}|x-y| > R$ holds for sets $A,B\subset \R^d$, then $(\YField (x))_{x\in A}$ and $(\YField (x))_{x\in B}$ are independent;
        \item\label{YMomentCondition}(finite moments) $\sup_{x\in [0,1]^d} |\YField (x)|$ is measurable and for some $\beta>4d$, 
        \[\E\Big[ \sup_{x\in [0,1]^d} |\YField (x)|^{\beta}\Big] < \infty.\]   
     \end{enumerate}
    \end{enumerate}
\end{enumerate}
The condition \ref{phiwDiffCondition} implies that $\|g_x\| \le \YField(x)$ for all $x\in \R^d$ since $\Xi_{v\to v}$ and $\Xi_{v\to v}^*$ are the identity maps. Also, note that we do not require finite range dependence on the field $g$ itself
in \ref{phiwDiffCondition}, we only need that it is dominated by a finite range field.

The last condition we need is uniform positive definiteness of the random Riemannian metric:
\begin{enumerate}[resume*]
    \item\label{uniformPositiveDefinite} There is $\lambda > 0$ such that $g_{x,\omega}(p,p) \ge \lambda |p|^2$ for all $p\in \R^d$, $x\in \R^d$, and $\omega\in \Omega$.
 \end{enumerate}

 Since for any path $\gamma$, $\frac{d}{dt}(\Xi_{v\to w}\gamma)_s = \Xi_{v\to w}\dot{\gamma}_s$, the transformed action introduced in \eqref{eq:AB_equality} can be rewritten for this model as
 \begin{equation}
    \label{eq:Bwvgamma1}
     B(w,v,\gamma) = \int_0^t\sqrt{ g^{w,v}_{\gamma_s}(\Xi_{v\to w}\dot{\gamma}_s,\Xi_{v\to w} \dot{\gamma}_s) }ds,\quad \gamma\in \SP_{0,Tv,t},
 \end{equation} 
 and the minimal transformed action $B^T(w,v)$ is defined according to~\eqref{eq:Bdef}. For $i=1,\dots, d$ and $x,p,v\in \R^d$ we let $h_x^i(p;v)$ denote $\partial_{w_i}g^{w,v}_x(p,p)\Big|_{w=v}.$
\begin{theorem}\label{th:main_riemannian} 
    Under assumptions \ref{stationaryPhiCondition}---\ref{uniformPositiveDefinite} and the notation defined above, all theorems of Section~\ref{sec:general-results} hold. 
\end{theorem}
\begin{remark}\label{rem:addtion-to_riemannian}  Under the assumptions of   \Cref{th:main_riemannian},
    \begin{enumerate}
    \item\label{positiveShapeRiemannianFPP} condition \eqref{cond:Lambda-cone} holds due to \ref{uniformPositiveDefinite};
     \item\label{RiemannianFPPDerivative} we can derive an expression for \eqref{eq:H-diff-of Lambda} in this example: for all $v\in \R^d$ and $w\in v+H$, setting $\gamma^T(v) := \gamma(0,Tv)$ the selection of minimizer in \ref{cond:cone}, we have
    \begin{equation}\label{eq:RiemannianFPPDerivative}
        \langle \nabla \Lambda(v),w\rangle = \lim_{T\to \infty}\frac{1}{T}\int_0^t\Big[\frac{1}{2} \sum_{i=1}^d w_i h^i_{\gamma_s^T(v)}(\dot{\gamma}_s^T(v)) + \frac{\langle v,\dot{\gamma}_s^T(v)\rangle}{|v|^2} g_{\gamma_s^T(v)}(w,\dot{\gamma}_s^T(v))\Big]ds.
    \end{equation}

    \end{enumerate}
\end{remark}

We will prove Theorem~\ref{th:main_riemannian} in Section~\ref{sec:proof-of-main-riemannian}. Part~\ref{positiveShapeRiemannianFPP} of 
\Cref{rem:addtion-to_riemannian} 
is a direct consequence of~\ref{uniformPositiveDefinite}. We will justify  part 
\ref{RiemannianFPPDerivative} in Section~\ref{sec:secondDerivRiemannianFPP}.

In \Cref{sec:examples-Riemannian} we present two examples of random Riemannian metrics that satisfy the conditions of this section.
\subsection{Examples of random Riemannnian metrics}
\label{sec:examples-Riemannian}

Let us give two examples of random Riemannian metrics satisfying the above requirements. 
Let $K\subset\R^d$ be a compact set and let $\Qq$ be a probability measure on the space 
\[
C_{K}(\R^d;\mathcal{M}^d_+)= \bigl\{f\in C_{loc}(\R^d;\mathcal{M}^d_+): \textrm{supp}(f)\subset K\bigr\}.
\] 
We let $\mathbf{N}$ be a Poisson measure on $\R^d\times C_{K}(\R^d;\mathcal{M}^d_+)$ with intensity measure given by $\Leb \otimes \Qq$. 
Here $\Leb$ is the Lebesgue measure on $\R^d$. 
In other words, this is a marked Poisson process with unit intensity on $\R^d$ and  i.i.d.\ marks 
distributed in $C_{K}(\R^d;\mathcal{M}^d_+)$ according to $\Qq$.

It is convenient to work with the canonical space $(\Omega,\Fc,\Prb)$ of locally finite Poisson point configurations on $\R^d\times C_{K}(\R^d;\mathcal{M}^d_+)$  equipped with topology of vague convergence. The role of $\omega$ of Section~\ref{sec:RiemannianFPPeg} is played by $\mathbf{N}$. 

For $y\in\R^d$, the translation $\theta^y$ on $\R^d$
also gives rise to a $\Prb$-preserving transformation $\theta^y_*$ of $\Omega$: 
each Poissonian point $(x_i,\varphi_i)$ is mapped into a Poisson point
$(\theta^y x_i,\varphi_i)$
of the Poisson Point Process~$\theta^y_*\mathbf{N}$. Here the transformation $\theta^y$ applies only to the base point $x_i$ in $\R^d$ but not to the mark $\varphi_i$. Equivalently, for all continuous functions $f$ with bounded support: 
\begin{equation}
    \int f(x,\varphi)(\theta_*^y\mathbf{N})(dx,d\varphi) = \int f(\theta^y x,\varphi)\mathbf{N}(dx,d\varphi).    
\end{equation}

\begin{example}\label{eg:FPPexample1}
    We can let
    \begin{equation}\label{eq:phiSumRepresentation}
        g_x = \int \varphi(x - y) \mathbf{N}(dy,d\varphi)+ \lambda I =\sum_{(x_i,\varphi_i)} \varphi_i(x-x_i)+\lambda I,
    \end{equation}
    where the summation extends over all Poissonian points $(x_i,\varphi_i)$.
    We must also assume that for some $\beta>4d$ 
    \begin{equation}\label{eq:varphiMoments}
        \Qq \|\varphi\|_{C^2}^\beta  < \infty.
    \end{equation}
\end{example}

\begin{example}\label{eg:FPPexample2}
    We can also consider a product version of \Cref{eg:FPPexample1}. Specifically, we can let 
    \begin{equation}\label{eq:phiProduct}
        g_x = \exp\Big(\int \varphi(x-y)\mathbf{N}(dy,d\varphi)\Big).
    \end{equation}
    We additionally require that there is some $C>0$ such that $\|\varphi\|_{C^2} \le C$ with probability one. 
\end{example}

\begin{theorem}\label{th:conditions-hold-for-Riemannian-models} 
    \Cref{eg:FPPexample1,eg:FPPexample2} satisfy our general conditions \ref{stationaryPhiCondition}---\ref{uniformPositiveDefinite}.
\end{theorem}

We will prove this theorem in \Cref{sec:appendix}. A crucial step is to define useful transformations  $\Xi_{v\to w}^*$, $w\in H$, on $\Omega$ preserving the distribution of the Poisson measure.
That can be done via the following identity for all continuous functions $f$
with bounded support: 
\begin{equation}
    \int f(x,\varphi)(\Xi_{v\to w}^*\mathbf{N})(dx,d\varphi) = \int f(\Xi_{v\to w}x,\varphi)\mathbf{N}(dx,d\varphi).    
\end{equation}
In other words, a marked Poissonian point $(x_i,\varphi_i)$ of the Poisson Point Process~$\mathbf{N}$ gives rise to
a point $(\Xi_{v\to w} x_i,\varphi_i)$ of  $\Xi_{v\to w}^*\mathbf{N}$, where the transformation $\Xi_{v\to w}$ applies only to the base point $x_i$ in $\R^d$ but not to the random kernel $\varphi_i$. Due to Lemma~\ref{lem:XiProperties}, these transformations preserve $\Prb$, the distribution of the Poisson measure. 
 
 \section{Example II: Broken line Poisson FPP.}
 \label{sec:exampleII}
 
 The goal of this section is to introduce another 
 family of random metrics on~$\R^d$ and state our main results for these models. In 
 Section~\ref{sec:proofs_broken-lines}, we 
 show that they follow from our general approach 
 by 
 checking the conditions of Section~\ref{sec:general-results}.

 We will work with a homogeneous Poisson point process of 
 constant intensity $1$ on $\R^d$. 
 Similarly to Section~\ref{sec:examples-Riemannian}, it is convenient to work with the canonical space $(\Omega,\Fc,\Prb)$ of locally finite Poisson point configurations on $\R^d$ equipped with topology of vague convergence. We usually denote elements of $\Omega$ by $\omega$. They can be viewed either 
 as locally finite point configurations or
 $\sigma$-finite Borel
 measure with values in $\N\cup\{0\}$ on bounded Borel sets.  
 For $x\in\R^d$ and $\omega\in\Omega$, we will write $x\in \omega$ if and only if $\omega(\{x\})=1$. The space $\Omega$ is equipped with the group of $\Prb$-preserving transformations 
 $\theta^x_*:\Omega\to\Omega$, $x\in\R^d$. Namely, for $\omega\in\Omega,\ x\in\R^d$, $\theta^x_*\omega$ is defined as the pushforward of the measure $\omega$ under the transformation $\theta^x$  defined in~\eqref{eq:theta}.

 For every $x\in\R^d$, we introduce a random variable
 \[
 F(x)=F_{\omega}(x)=1- \omega(\{x\})=\begin{cases}
     1,& x\notin \omega,\\
     0,& x\in \omega.
 \end{cases}
 \]
 It will serve as a cost for a path to contain $x$. In other words, no cost will be
 associated with Poissonian points, while all the other points (we will call them {\it penalty points})
 add cost~$1$ to a path.  
 
 We will also need a cost function, or Lagrangian, $L$. Throughout this section, we will assume that $L$ satisfies the following conditions:
 \begin{enumerate}[label=(D\arabic*), ref=\rm{(D\arabic*)}] 
     \item\label{asm:L-convex}
  $L:\R^d\to[0,\infty)$ is a convex function such that 
 $L(x)=L(-x)$ for all $x\in\R^d$ and 
 $L(x)=0$ if and only if $x=0$. 
 \item \label{eq:L-C2} $L\in C^2(\R^d)$. 
 \item \label{asm:lower-L-at-0} There is 
 $c>0$ such that
     $L(x)\ge c|x|^2,\quad |x|\le 1.$
 \item \label{asm:superlinear_L}
     $\lim_{x\to\infty}\frac{L(x)}{|x|}=+\infty.$
 \end{enumerate}
 
 We first introduce random action for discrete paths, i.e., 
 sequences of points   
 $\gamma=(\gamma_0,\ldots,\gamma_n)$ in $\R^d$.  Points $\gamma_k\in\R^d$, $k=0,1,\ldots,n$, are called vertices of $\gamma$.
 
 For $x,y\in\R^d$ and $n\in\N$, we define
 \begin{align}
     \notag
 \Paths_{x,y,n}&=\Big\{\gamma:\{0,1,\ldots,n\}\to \R^d :\  \gamma_0=x,\ \gamma_n=y\Big\},
 \\ 
 \label{eq:paths_stars}
 \Paths_{x,y,*}&=\bigcup_{n\in\N} \Paths_{x,y,n},\\
 \notag
 \Paths_{*,*,n}&=\bigcup_{x,y\in\R^d} \Paths_{x,y,n}, \quad \mathrm{etc.}
 \end{align}
 The set $\Paths_{*,*,n}$ can be identified with $\R^{(n+1)d}$ and equipped with the Euclidean topology. The set $\Paths=\Paths_{*,*,*}=\cup_{n}\Paths_{*,*,n}$ is equipped with the disjoint union topology. We can embed $\Paths_{x,y,n}$ into $\SP_{x,y,n}$ by considering the piecewise linear interpolations of paths in $\Paths_{x,y,n}$, see \Cref{sec:proofs_broken-lines} for details.

 For a path $\gamma\in \Paths_{*,*,n}$, we  define its action by
 \begin{align}
     \label{eq:Poisson-action}
     A_{\omega}(\gamma)&=\sum_{i=0}^{n-1} L(\Delta_i \gamma)
     +\frac{1}{2}\sum_{i=0}^{n-1} (F_{\omega}(\gamma_i)+F_{\omega}(\gamma_{i+1}))
     \\
     \notag
     &=\sum_{i=0}^{n-1} L(\Delta_i \gamma)
     +\frac{1}{2}F_\omega(\gamma_0)+\sum_{i=1}^{n-1} F_{\omega}(\gamma_i)+ \frac{1}{2} F_{\omega}(\gamma_{n}),
 \end{align}
 where $\Delta_i \gamma =\gamma_{i+1}-\gamma_i$. 
 
 Our results apply to various similar definitions, but we choose this one because it is additive under concatenation of paths and invariant under path reversal, see the proof of Lemma~\ref{lem:metric} below.

 For distinct $x,y\in\R^d$ and every $\omega\in\Omega$, we can define 
 \begin{equation}
     \label{eq:def_random_metric1}
 \Amin_\omega(x,y)=\inf_{\gamma\in\Paths_{x,y,*}} A_\omega(\gamma).
 \end{equation}

 We also set $\Amin_\omega(x,x)=0$ for all $x\in\R^d$. This is compatible 
 with definition \eqref{eq:Poisson-action} with $n=0$ and $\gamma=(x)\in \Paths_{x,x,0}$.
 If $\gamma\in \Paths_{x,y,*}$ satisfies
 $\Amin_\omega(x,y)=A_\omega(\gamma),$   
 then we call $\gamma$ a geodesic between $x$ and~$y$ under $\omega$.

 \begin{lemma}\label{lem:metric}
     For all $\omega\in\Omega$, $\Amin_\omega$ is a finite metric on $\R^d$. 
 \end{lemma}
 \begin{proof}
 For any two points $x,y\in\R^d$, let $\gamma$ be a path in $\Paths_{x,y,*}$ satisfying $|\gamma_{i+1} - \gamma_i| \le 1$ and having at most $\lceil|x-y|\rceil+1$ steps, and so
 \begin{equation}
     \label{eq:bound_on_rho_1}
     \Amin_\omega(x,y)\le A(\gamma) \le (\lceil |x-y|\rceil+1)(\sup_{|x|\le 1}L(x) + 1)<\infty.
 \end{equation}
 The symmetry of $\Amin_\omega$ follows from the invariance under path reversal:  for all $\omega$, $n$, and $\gamma\in\Paths_{*,*,n}$,
 we have 
     $A_\omega(\gamma_n,\gamma_{n-1},\ldots,\gamma_1,\gamma_0)=A_\omega(\gamma)$.
 
 To prove the triangle inequality, we introduce path concatenation: for any point $x\in\R^d$, if 
 $\gamma=(\gamma_0,\gamma_{1},\ldots,\gamma_{n-1},\gamma_n)\in \Paths_{*,x,*}$, and
 $\psi=(\psi_0,\psi_{1},\ldots,\psi_{m-1},\psi_m) \in \Paths_{x,*,*}$ for some $x\in\R^d$
 we define
 \[\gamma \psi=
 (\gamma_0,\gamma_1,\ldots,
 \gamma_{n-1},x,\psi_1,\ldots,
 \psi_{m-1},\psi_{m}).
 \] 
 It follows that if concatenation of paths $\gamma$ and
 $\psi$ is well-defined, then, for all $\omega$,
 \begin{equation}
     A_\omega(\gamma\psi)=A_\omega(\gamma)+A_\omega(\psi),
     \label{eq:concatenation}
 \end{equation}
 a property similar to \ref{cond:subbadd} for continuous paths, implying the triangle inequality.
 
 The relation $\Amin_\omega(x,y)>0$ for distinct $x,y\in\R$ is also easy to see. Namely, paths containing at least one penalty point have action at least $1/2$, and a path $\gamma$ containing no penalty points, contains only Poisson points, so its action is bounded
 below, due to \ref{asm:lower-L-at-0}, by $c(\Delta_\omega^2(x)\wedge 1)$, where 
 $\Delta_\omega(x)$ is the Euclidean distance from~$x$ to the closest Poissonian point distinct from $x$.
\end{proof}

 \begin{theorem}\label{thm:main-on-broken-lines} Under assumptions \ref{asm:L-convex}--\ref{asm:superlinear_L}, all theorems of Section~\ref{sec:general-results} hold.
 \end{theorem}    
 \begin{remark}\label{rem:brokenLineRemark} Under the assumptions of Theorem~\ref{thm:main-on-broken-lines}:
    \begin{enumerate}
        \item\label{item:shapeFcnPositiveBrokenLine} The shape function positivity condition~\eqref{cond:Lambda-cone} holds. Thus, according to 
        Theorem~\ref{thm:limit-shape-diff}, the boundary of the limit shape is diffeomorphic to a sphere.
        \item\label{item:BrokenLineDerivative} For all $v\in \R^d$ and $w\in v+H$, setting $\gamma^T(v) := \gamma(0,Tv)$ to be the selection of minimizer in \ref{cond:cone}, we have
        \begin{equation}\label{eq:BrokenLineDerivative}
            \langle \nabla \Lambda(v),w\rangle = \lim_{T\to \infty}\frac{1}{T|v|^2}\sum_{i=0}^{n-1} \langle \nabla L(\Delta_i \gamma^T(v)),w\rangle\langle v,\Delta_i \gamma^T(v)\rangle.
        \end{equation}
    \end{enumerate}
 \end{remark}
 We prove \Cref{thm:main-on-broken-lines} in Section~\ref{sec:proofs_broken-lines}, where we interpret our model in terms of continuous
 paths from~$\SP$, check the conditions 
 of Section~\ref{sec:general-results}, and apply 
 Theorems~\ref{thm:shapeFunction}---\ref{thm:limit-shape-diff}. Part \ref{item:shapeFcnPositiveBrokenLine} of \Cref{rem:brokenLineRemark} will follow from \Cref{lem:action_at_least_linear}, and part \ref{item:BrokenLineDerivative} will follow from computations in \Cref{sec:proofs_broken-lines}.

 \section{The Directed Setting}\label{sec:directed}
  In Sections \ref{sec:exampleI} and \ref{sec:exampleII}, we discussed examples of FPP type involving no restrictions on admissible path directions.
  The goal of this section is to explain how the directed 
  setting of \cite{bakhtin2023differentiability}, where the time coordinate plays a distinguished role, also fits the general framework of the present paper, although we used slightly different notations and definitions in that paper. 
  
Let $\SP^\uparrow$ denote the subset of $\SP$ given by paths $\gamma$ in $\R^d$ satisfying $\langle \dot{\gamma}, e_1\rangle \equiv 1.$ In this directed setting, the first coordinate is interpreted as time. For $x\in \R^d$, we let~$x^\uparrow$ be the vector composed of the remaining $d-1$ coordinates of~$x$. For $\gamma\in \SP^\uparrow$, the path $\gamma^\uparrow$ is defined by $(\gamma^\uparrow)_s=\gamma^\uparrow_s$. In \cite{bakhtin2023differentiability}, the action was given by
\begin{equation}\notag
    A_\omega(\gamma) = \int_{0}^t L(\dot{\gamma}_s^\uparrow)ds + \int_{0}^t F_\omega(\gamma_s)ds, \quad\gamma\in \SP^\uparrow \cap \SP_{*,*,t},
\end{equation}
for $L:\R^{d-1}\to \R$ a convex function and $F$ is a random twice differentiable function, both satisfying certain assumptions. 

To apply the general set-up of \Cref{sec:generalSetUp}, we can set the action to be infinite on all other paths and define the cone $\Cone$ of condition \ref{cond:cone} to be $\{x\in \R^d\,:\,\langle x,e_1\rangle > 0\}$. To ensure condition \ref{cond:setup_diff}, we let $H = \{(0,x)\,:\,x\in \R^d\}$ and,  for $v,w\in\Cone$, we define
\[\Xi_{v\to w} x = x + (w-v)\frac{\langle x,e_1\rangle}{\langle v,e_1\rangle},\quad x\in\R^d.\]
 Note that this shear transformation satisfies $\Xi_{v\to w} T v = Tw$ for all $T$.
It can be lifted to a map on $\SP$ by setting $(\Xi_{v\to w}\gamma)_s = \Xi_{v\to w}\gamma_s$.
If  $w\in v+H$  (i.e., the time coordinates of $v$ and~$w$ coincide), the map $\Xi_{v\to w}^*$ on $\Omega$ is chosen as a measure preserving transformation such that the derivatives of $F_{\Xi_{v\to w}^*\omega }(\Xi_{v\to w}x)$ with respect to $w$ 
allow for a bound in terms of a stationary, finite dependence range stochastic process with sufficiently high moments. This implies the bound on $B_\omega(w,v,\gamma)$ required in~\ref{cond:M_infty} since

 in this setting 
\begin{equation}\notag
    B_\omega(w,v,\gamma) = A_{\Xi^*_{v\to w}}(\Xi_{v\to w} \gamma) = \int_0^t L(\dot{\gamma}^\uparrow_s + (w-v)^\uparrow)ds + \int_0^t F_{\Xi_{v\to w}^*\omega}(\Xi_{v\to w}\gamma_s)ds
\end{equation}
for $\gamma\in \SP^\uparrow.$
We refer the reader to \cite{bakhtin2023differentiability} for technical details.

 \section{Proofs of results from Section~
 \ref{sec:general-results}}\label{sec:uniformConvergence}
 In this and following sections, we give proofs of results stated in
Sections~\ref{sec:general-results}--\ref{sec:exampleII}.

 \subsection{Proof of Theorem~\ref{thm:shapeFunction}}
     Since $\theta^z$ maps $\SP_{x,y,*}$ bijectively to $\SP_{x+z,y+z,*},$ condition \ref{cond:skew-invariance}
     implies 
     \begin{equation}
        \Amin_\omega(x,y) = \Amin_{\theta_*^z \omega}(x+z,y+z).
     \end{equation}
     If $\gamma_1\in \SP_{0,x,*}$ and $\gamma_2\in \SP_{x,x+y,*}$, then $\gamma_1\gamma_2\in \SP_{0,x+y,*}.$ We deduce that for all $x,y\in \R^d$
     \begin{equation}\label{eq:subadditiveProperty}
        \Amin_\omega(0,x+y) 
        \le \inf_{\gamma_1\in \SP_{0,x,*},\,\gamma_2\in \SP_{x,x+y,*}}(A(\gamma_1) + A(\gamma_2))
         \le \Amin_\omega(0,x) + \Amin_{\theta_*^{-x}\omega}(0,y).
     \end{equation}
     If $v\ne 0$, 
     combining this with \eqref{eq:expected_action_finite} of condition 
     \ref{cond:cone}, we obtain existence of the limit in~\eqref{eq:conv_to_Lambda} from Kingman's Subadditive Ergodic Theorem (see Theorem 5 in \cite{Kingman:MR0356192}). The fact that $\Lambda$ is deterministic follows from ergodicity of $\theta_*.$ If $v=0$, then $|\Amin_\omega(0,T0)|=|\Amin_\omega(0,0)|<\infty$ due to \eqref{eq:expected_action_finite}, which implies $\Lambda(0)=0$.
 
     To prove convexity of $\Lambda$, we note that if  $z = \alpha x + (1-\alpha)y\in \Cone$ for some $x,y\in\Cone$, $\alpha\in(0,1)$ then \eqref{eq:subadditiveProperty} implies 
     \[\frac{1}{T}\Amin_\omega(0,Tz) \le \frac{1}{T}\Amin_\omega(0,T\alpha x) + \frac{1}{T}\Amin_{\theta_*^{-T\alpha x}\omega}(0,T(1-\alpha)y).\]
     The left-hand side converges almost surely to $\Lambda(z)$, and  
     the right-hand side converges in probability to $\alpha\Lambda(x) + (1-\alpha)\Lambda(y)$ as $T\to \infty$. Hence, $\Lambda$ is convex. 
 
     The property $\Lambda(sv) = s \Lambda(v)$ for positive $s$ follows directly from \eqref{eq:conv_to_Lambda} because $\frac{1}{T}\Amin^T(0,Tsv) = s\frac{1}{Ts}\Amin(0,Tsv)$.

 \subsection{Proof of Theorem~\ref{thm:uniformConvergenceShape}}
 Our argument closely follows that of Theorem~2.16 in~\cite{AuffingerDamronHanson_50Years:MR3729447}. We need an auxiliary lemma first.
 \begin{lemma}\label{lem:goodApproxLemma}
     For all $w\in \Cone^{\circ}$ and all $\epsilon > 0$ there are $v^-,v^+\in \Cone^{\circ}$ and $\delta > 0$ such that $|v^\pm-w|\le \epsilon$ and the following estimates hold with probability $1$:
     \begin{equation}\label{eq:AApproxUpperBound}
         \limsup_{T\to \infty}\frac{1}{T}\sup_{w'\in \Ball(Tw,T\delta )}\Amin(0,w') \le \Lambda(v^-) + \epsilon,
     \end{equation}
     \begin{equation}\label{eq:AApproxLowerBound}
         \liminf_{T\to \infty}\frac{1}{T}\inf_{w'\in \Ball(Tw,T\delta )}\Amin(0,w')  \ge  \Lambda(v^+) - \epsilon.
     \end{equation}
 \end{lemma}
 \begin{proof}
 We will first establish \eqref{eq:AApproxUpperBound}.
     
     We will take $\epsilon' > 0$ (to be specified later) satisfying $\epsilon' \le \epsilon.$  Take $\delta > 0$, $v^-\in \Cone^{\circ}$ and $\Cone^-\subset \Cone$ be such that $\Cone^-$ is a cone properly contained
     in $\Cone$ (see our definition of proper containment just before \ref{cond:for-finiteness}) such that $\Ball(w, \delta ) \subset v^- + \Cone^-$ and $|w-v^-| < \epsilon'.$ Note that this implies that $\Ball(R w, R \delta) \subset r v^- + \Cone^-$ for all $r$ and~$R$ satisfying  $0< r \le R.$

     Due to \eqref{eq:linearGrowthProb} of condition \ref{cond:for-finiteness}, ergodicity with respect to $\theta_*$, 
     and skew-invariance condition~\ref{cond:skew-invariance}, we obtain that
     with probability one there exists an increasing sequence $(n_k)_{k\in \N}$ such that $n_k\to \infty$ as $k\to \infty,$ $n_{k+1}/n_k\to 1$ as $k\to \infty,$ and
     \begin{equation}\label{eq:nkSequenceReq}
        \Amin(n_k v^-,w') \le \kappa |w' - n_k v^-|,\quad \forall w'\in n_k v^- + \Cone^-.
     \end{equation}

     Additionally, almost surely $\frac{1}{T} \Amin(0,T v^-)\to \Lambda(v^-)$ as $T\to \infty.$ Let us now 
     consider~$\omega$ satisfying these two conditions.
 
     By concatenating paths from $0$ to $v^-$ and from $v^-$ to $w'$, we deduce from \eqref{eq:nkSequenceReq} that for all $k\in \N$ and all $w'\in n_k v^- + \Cone^-,$
     \begin{equation}\label{eq:AwPrimeBound}
        \Amin(0,w') \le \Amin(0,n_kv^-) + \Amin(n_k v^- , w') \le \Amin(0,n_kv^-) + \kappa |w' - n_kv^-|.
     \end{equation}
     For $T>0$ sufficiently large, let $k(T)$ be such that 
     \[ n_{k(T)} \le T \le n_{k(T)+1}.\]
     For all $w' \in \Ball(Tw, T\delta )\cap \Cone$, \eqref{eq:AwPrimeBound} and $|w' - Tw|<T\delta$ imply
     \[\Amin(0,w') \le \Amin(0,n_{k(T)}v^-) + \kappa (|Tw - n_{k(T)} v^-| + T \delta ). \]
     Since $n_{k(T)} / T\to 1$ as $T\to \infty$ and $|w-v^-|<\epsilon'$, the right-hand side of this inequality is bounded above for sufficiently large $T$ by 
     \[\Amin(0,n_{k(T)} v^-) + \kappa (2 T \epsilon' + T \delta ).\]
     Taking $\epsilon' < \epsilon/(4\kappa)$ and $\delta < \epsilon / 2$, we obtain 
     \[\sup_{w'\in \Ball(T w, T\delta )}\Amin(0,w') \le \Amin(0,n_{k(T)} v^-) + T \epsilon.\]
     Dividing by $T$ and taking $T\to \infty$ establishes \eqref{eq:AApproxUpperBound} on an event of full probability.
     
     The argument for the claim \eqref{eq:AApproxLowerBound} is almost identical to the preceding argument but using \eqref{eq:linearGrowthProb2} in place of \eqref{eq:linearGrowthProb}, so we will only give a sketch of the proof. Choose $\epsilon'' \le \epsilon$ and $v^+ \in w + \Cone'$ satisfying $|v^+ - w| < \epsilon''$. Let $\delta' > 0$ and $\Cone^+$ be a cone properly contained in $\Cone$ such that $\Ball(w,\delta) \subset v^+ - \Cone^+.$   Almost surely there is an increasing sequence $n_k\to \infty$ satisfying $n_{k+1} / n_k\to 1$ and such that 
     \[\Amin(w',n_k v^+) \le \kappa|w' - n_kv^+|,\quad \forall w'\in n_k -\Cone^+.\]
     If $k=k(T)$ is such that $n_{k(T)}\le T\le n_{k(T)+1}$, then for all $w'\in \Ball(Tw,T\delta),$
     \begin{align*}
        \Amin(0,n_k v^+) \le \Amin(0,w') + \Amin(w',n_kv^+)& \le \Amin(0,w') + \kappa(|n_k v^+ - Tw| + T\delta)\\&\le \Amin(0,w') + T\epsilon
     \end{align*}
     for $\epsilon''$ and $\delta$ sufficiently small. We can then conclude that \eqref{eq:AApproxLowerBound} holds on a full measure set by dividing by $T$ and taking $T\to \infty.$
 \end{proof}

 \begin{proof}[Proof of \Cref{thm:uniformConvergenceShape}]
 
     First note that \eqref{eq:linearGrowthProb} implies $\Lambda(v)> -\infty$ for all $v\in \Cone^{\circ}.$ Indeed, if $\Cone'$ is a cone properly contained  in $\Cone$ that contains $v$, then 
     \[\Lambda(v) = |v|\lim_{T\to \infty}\frac{\Amin^T(v)}{|v|T} \ge - |v|\sup_{x\in \Cone',\,|x|>1}\frac{|\Amin(0,x)|}{|x|}>-|v|\kappa > -\infty\]
     with positive probability. Since $\Lambda(v)$ is nonrandom, we have $\Lambda(v) > -\infty$. 
     
     For every $w\in \Cone^{\circ}$ and $\epsilon > 0$, let $v^-(w,\epsilon),v^+(w,\epsilon)\in \Cone^{\circ}$ and $\delta(w,\epsilon)$ be such that the conclusions of \Cref{lem:goodApproxLemma} hold.

     Since $\bigcup_{w\in \Cone} \Ball(w,\delta(w,\epsilon))$ is an open cover of $\Cone^{\circ},$ for every $\epsilon>0,$ there is a countable set of tuples $(w_k(\epsilon), v_k^-(\epsilon),v_k^+(\epsilon),\delta_k(\epsilon))$ indexed by $k\in \N$ such that the conclusions of \Cref{lem:goodApproxLemma} hold for each tuple and such that $\bigcup_{k\in \N}\Ball(w_k(\epsilon),\delta_k(\epsilon))$ is an open cover of $\Cone^{\circ}.$

     Define the event $\Omega(k,\epsilon)$ as
     \begin{align*}
         \Omega(k,\epsilon) = \Big\{\omega\,:\, \textrm{\eqref{eq:AApproxUpperBound} and \eqref{eq:AApproxLowerBound} hold for }(w_k(\epsilon), v_k^-(\epsilon),v_k^+(\epsilon),\delta_k(\epsilon))\Big\}
     \end{align*}
     Finally, define the event
     \[\Omega_0 = \bigcap_{m\in \N}\bigcap_{k\in \N} \Omega(k,m^{-1}).\]
     By \Cref{lem:goodApproxLemma}, $\Prb(\Omega_0) = 1.$ We will prove that \eqref{eq:uniformConvergence} holds for all $\omega\in \Omega_0.$
 
     Suppose $\omega\in\Omega_0$, let $\epsilon > 0$, and let $K$ be a compact subset of $\Cone^{\circ}$. Let $m\in \N$ be such that $m^{-1} < \epsilon$ and such that if $|v-w|<m^{-1}$ and $w\in K,$ then 
     \begin{equation}\label{eq:epsilonCloseLambda}
         |\Lambda(v)-\Lambda(w)| < \epsilon.
     \end{equation}

     By compactness of $K$, we can find indices $i_1,\dots,i_k$ such that 
     \begin{equation*}
         K\subset \bigcup_{\ell=1}^k \Ball(w_{i_\ell}(m^{-1}),m^{-1}).
     \end{equation*}
     If $w \in \Ball(w_{i_\ell}(m^{-1}), m^{-1})$ for some $\ell\in\{1,\ldots,k\}$, then
     \begin{align*}
         \frac{1}{T}\Amin^T(w) - \Lambda(w)  \le \frac{1}{T}(\Amin^T(w) - \Lambda(v^-_{i_\ell}(m^{-1}))) + |\Lambda(v_{i_\ell}^-(m^{-1})) - \Lambda(w)|.
     \end{align*}
     Applying \eqref{eq:AApproxUpperBound} and \eqref{eq:epsilonCloseLambda}, we obtain
     \begin{align*}
         \limsup_{n\to \infty}\sup_{w\in B(w_{i_\ell}(m^{-1}), m^{-1})}\Big[\frac{1}{T}\Amin^T(w) - \Lambda(w) \Big] \le 2 \epsilon.
     \end{align*}
     The above holds for all $\epsilon > 0$ and all $\ell=1,\dots, k$. Therefore,
     \[\limsup_{T\to \infty}\sup_{w\in K} (\frac{1}{T} \Amin^T(w) - \Lambda(w)) \le 0.\]
     The reverse inequality can be proven similarly using \eqref{eq:AApproxLowerBound}, and  \eqref{eq:uniformConvergence} follows. 
 \end{proof}

 \subsection{Proof of Theorem~\ref{thm:limit-shape}}\label{sec:proof-limit-shape}
     Fix $\epsilon > 0$ and a compact set $K\subset \Cone^{\circ}$. The set $K':= ((1+\epsilon)\rball_\Lambda)\cap K$ is a compact subset of $\Cone^{\circ}$.    Thus, \Cref{thm:uniformConvergenceShape} implies that if $\omega\in\Omega_0$, then
     \begin{equation}
         \label{eq:Delta_T-to0}
     \lim_{T\to\infty}\Delta_T=0,
     \end{equation}
     where
     \[\Delta_T=\sup_{w\in K'}|\Lambda(w) - \frac{1}{T}\Amin(0,Tw)|.\]
     If $v\in ((1-\epsilon)\rball_\Lambda)\cap K$, then  $\Lambda(v) \le 1-\epsilon$ due to
     $1$-homogeneity of $\Lambda$. Therefore,
     \[\frac{1}{T}\Amin(0,Tv) \le \Lambda(v) + \Delta_T\le 1-\epsilon + \Delta_T,\]
    If $\Delta_T < \epsilon$ then $v\in \frac{1}{T}\rball_\omega(T)$. Additionally, if $v\in \frac{1}{T}\rball_\omega(T)\cap K$, then
     \[\Lambda(v) \le \frac{1}{T}\Amin(0,Tv)+ \Delta_T\le 1 + \Delta_T,\]
     implying that $v\in K'$ if $\Delta_T\le \epsilon.$ Therefore, \eqref{eq:Delta_T-to0} implies that 
 there is $T_0 > 0$ such that the inclusion \eqref{eq:local-convergence-of-shapes} (with $N = \rball_\Lambda$) holds and thus 
 Theorem~\ref{thm:limit-shape} is established.

 To prove the claim made in Remark~\ref{rem:limit-shape-strong}, we note that if $\Cone=\R^d$ and $\Lambda(v)>0$ for all $v\ne 0$, then
 $\rball_\Lambda$ is compact and contained in $\Cone^{\circ}$.
Thus, \eqref{eq:shapeConvergence-strong} follows from \eqref{eq:shapeConvergenceCompact} since we can take $K=(1+\epsilon)\rball_\Lambda$ in the definition of local convergence.

 \subsection{Proof of \Cref{lem:diffLemma}}
     Differentiability of $f$ relatively to $H$ means that there is a vector $F\in H$ such that
     \begin{equation}
         \label{eq:F-is-grad-f}
     f(w) = f(v) + \langle F, w-v\rangle + o(|w-v|),\quad H(\delta)\ni w\to v.
     \end{equation}

     We let $e_i$ be the basis vector in $\R^d$ with 1 in the $i$th coordinate and $0$ elsewhere. Let $h_1,\dots, h_{d-1}$ be an orthonormal basis for $H$ and let $\mathbf{H}$ be the linear map satisfying $\mathbf{H} e_i = h_i$ for $i=1,\dots, d-1$ and $\mathbf{H} e_d = v.$ Since $v\notin H,$ $\mathbf{H}$ is invertible. Note that since $u = \mathbf{H}\mathbf{H}^{-1}u,$ the vector $\mathbf{H}^{-1}u$ is the representation of $u$ in the basis given by $(h_1,\dots, h_{d-1},v).$
 
     Define the map 
     \[s(w) = \frac{1}{\langle e_d, \mathbf{H}^{-1} w\rangle},\quad w\in \R^d\setminus\{0\}.\]
     We claim that $s(w)w \in v + H$ for all $w\in \R^d\setminus\{0\}$. We have $u\in H$ if and only if $\langle e_d, \mathbf{H}^{-1}u\rangle = 0,$ because 
     \[u = \mathbf{H} \mathbf{H}^{-1} u =  \sum_{i=1}^{d-1}\langle e_i, \mathbf{H}^{-1}u\rangle h_i+ \langle e_d, \mathbf{H}^{-1}u\rangle v.\]
     Since $\mathbf{H}^{-1}v = e_d$, the claim $s(w)w\in v+H$ follows from 
     \[\langle e_d, \mathbf{H}^{-1}(s(w)w - v) \rangle = \frac{1}{\langle e_d, \mathbf{H}^{-1} w\rangle}\langle e_d, \mathbf{H}^{-1}w \rangle - \langle e_d, \mathbf{H}^{-1}v \rangle=0.\]

     We must show that $f(w) - f(v)$ is $o(|w-v|)$-close to a linear map, as $w\to v.$ Define $I_1$ and $I_2$ in the following way:
     \begin{align*}
         f(w)  - f(v)&= (f(w) - f(s(w)w)) + (f(s(w)w) - f(v)) = I_1 + I_2.
     \end{align*}
     Using differentiability of $s$, the fact that $s(v) = 1$, and continuity of $f$, we obtain
     \begin{multline*}I_1 = f(w)(1 - s(w)) = -f(w) \langle \nabla s(v),w-v\rangle + o(|w-v|) 
         \\ = -f(v)\langle \nabla s(v),w-v\rangle + o(|w-v|).
     \end{multline*}
     Additionally, since $s(w)w - v\in H,$ we can use \eqref{eq:F-is-grad-f} to write
     \[I_2 = \langle F, s(w)w - v\rangle + o(|s(w)w - v|).\]
     We have $s(w)w - v = (s(w)-1)w + w-v$, so $o(|s(w)w - v|) = o(|w-v|)$. Therefore,
     \begin{align*}
         I_2 &= \langle F, s(w)w - v\rangle + o(|w - v|) \\
         & = \langle F,w-v\rangle + (s(w) - 1)\langle F,w\rangle + o(|w-v|)\\
         & = \langle F,w-v\rangle + \langle \nabla s(v),w-v\rangle\langle F,v\rangle + o(|w-v|).
     \end{align*}
     In sum, 
     \[f(w) - f(v) =  -f(v)\langle \nabla s(v),w-v\rangle + \langle F,w-v\rangle + \langle \nabla s(v),w-v\rangle\langle F,v\rangle + o(|w-v|).\] Hence, $f$ is differentiable at $v$ with gradient given by $(\langle F,v\rangle-f(v))\nabla s(v) + F$.
 
 \subsection{Proof of Lemma~\ref{lem:deterministicDifferentiability}}
 In the below we continue to fix $v\in \Cone^\circ,$ $\delta > 0$ and define $H(\delta) = (v+H)\cap \Cone^\circ\cap \Ball(v,\delta).$ For a function $f:H(\delta)\to \R$ we define the subdifferential to be 
 \[\partial^\vee_H f(x) = \{\xi\in H\,:\,\forall y\in H,\,f(y)-f(x)\ge \langle \xi,y-x\rangle\}.\]
 If $f$ is convex, then $\partial^\vee_H f(x)$ is nonempty. A convex function is differentiable on $v+H$ at $v$ if and only if $\partial^\vee_H f(v)$ is a one-point set. 
     Let $\mathcal{G} \subset H\cup \{\infty\}$ denote the set of limit points of the sequence $(\xi_n)_{n\in \N}$ and let $\xi^*\in\mathcal{G}$ be a limit point of some subsequence $(\xi_{n_k})_{k\in \N}$. First we rule out the case $\xi^* = \infty.$ If $\xi_{n_k}\to \infty$ we may, by taking a further subsequence, assume that $\xi_{n_k}/|\xi_{n_k}|$ converges to some vector $\xi^*_\infty.$ Take a vector $w\in \mathcal{D}$ such that $\langle \xi^*_\infty,w-v\rangle < 0.$ Then, 
 $\langle \xi_{n_k}/|\xi_{n_k}|,v-w\rangle >0$ for sufficiently large $k$ and by \eqref{finiteDerivativeInequality} and \eqref{eq:f_n-to-f}, $|\xi_{n_k}|\langle \xi_{n_k}/|\xi_{n_k}|,v-w\rangle$ is bounded above. It follows that the sequence $(|\xi_{n}|)_{n\in \N}$ is bounded.
     
     Taking $k\to \infty$ in the inequality 
     \[f_{n_k}(w) - f_{n_k}(v) \le \langle \xi_{n_k}, w-v\rangle + h(w),\]
      we obtain 
     \begin{equation}\label{fUpperDifferential}
         f(w) - f(v) \le \langle \xi^*, w-v\rangle + h(w)
     \end{equation}
     for all $w\in \mathcal{D}$,
     Let $p\in \partial^\vee f(v).$ Then, the inequality
     \[f(w) - f(v) \ge \langle p, w-v\rangle\]
     along with \eqref{fUpperDifferential} implies 
     \begin{equation}\label{gStarInequality}
         0 \le \langle \xi^* - p, w-v\rangle + h(w)
     \end{equation}
     for all $w\in \mathcal{D}$. Let $\zeta \in H$ satisfy $|\zeta| = 1$. Since $\mathcal{D}$ is dense in $H(\delta),$ there is a sequence $(w_m(\zeta))_{m\in \N}$ in $\mathcal{D}$ converging to $v$ such that 
     \[\lim_{m\to \infty}\frac{w_m(\zeta)-v}{|w_m(\zeta)-v|} = \zeta.\]
     Inequality \eqref{gStarInequality} then implies 
     \[0 \le \langle \xi^* - p, \zeta\rangle.\]
     Repeating this procedure with $-\zeta$ gives us 
     \[0 = \langle \xi^*-p, \zeta\rangle.\]
     Since the above holds for all $\zeta\in H$ satisfying $|\zeta| = 1$, we have $\xi^* = p.$ But $\xi^*\in \mathcal{G}$ and $p\in \partial^\vee f(v)$ were arbitrary, and so, in fact, $\lim_{n\to \infty}\xi_n$ is well-defined and
     \[\mathcal{G} = \partial^\vee f(v) = \big\{\lim_{n\to \infty}\xi_n\big\}.\]
     Hence, $f$ is differentiable at $v$ with derivative equal to $\lim_{n\to \infty} \xi_n.$

\section{Proof of Theorem~\ref{th:main_riemannian}}
\label{sec:proof-of-main-riemannian}
\Cref{th:main_riemannian} will follow once we prove that the action and minimal action described in \Cref{sec:RiemannianFPPeg} satisfy the conditions outlined in \Cref{sec:general-results}.

Conditions \ref{cond:subbadd} --- \ref{cond:cone} are checked in Section~\ref{sec:checking-A}. Conditions \ref{cond:setup_diff}--\ref{cond:M_infty} 
and Part~\ref{RiemannianFPPDerivative} of Remark~\ref{rem:addtion-to_riemannian}
are checked in
Section~\ref{sec:secondDerivRiemannianFPP}.
Condition~\ref{cond:for-finiteness} is checked in 
Section~\ref{sec:checking-cond-for-finiteness}. Proofs of some auxiliary results are given in Section~\ref{sec:proofs_of_lemmas}.

\subsection{Conditions \ref{cond:subbadd} --- \ref{cond:cone}}\label{sec:checking-A}
The map $A:\Omega\times \SP\to \R$ defined in \eqref{eq:AgammaDefFPP1} is measurable due to measurability of $g$. 
A stronger (strictly additive) version of the subadditivity condition \ref{cond:subbadd} follows due to additivity of integrals. The skew invariance condition 
\ref{cond:skew-invariance} is implied by assumption \ref{stationaryPhiCondition} and the fact that $\frac{d}{dt}(\theta^z \gamma)_s = \theta^z\dot{\gamma}_s$.

The following lemma verifies the measurability of the minimal action 
required in condition~\ref{cond:measurability}
as well as the existence of a measurable selection of minimizer required in condition~\ref{cond:cone}. See \Cref{sec:checkingConditionsFPP} for the proof.
\begin{proposition}\label{prop:measurableSelectionFPP}
    \begin{enumerate}
    \item \label{it:min-seclection}
    There is a measurable mapping $\gamma^*$ from $\R^d\times \R^d\times \Omega$ to $\SP$ such that $\gamma^*(x,y,\omega)\in \SP_{x,y,*}$ and
    \begin{equation}\label{eq:gammaXYDef}
        A(\gamma^*(x,y,\omega)) = \Amin_\omega(x,y)
    \end{equation}
    for all $(x,y,\omega).$ 
    Additionally, $\gamma^*$ can be chosen so that
    \begin{equation}\label{eq:gammaStarSpeed}
        g_{\gamma^*_s(x,y)}(\dot{\gamma}^*_s(x,y),\dot{\gamma}^*_s(x,y)) = 1,\quad s\in[0,t].
    \end{equation}
    \item  $\Amin:\Omega\times \R^d\times \R^d\to\R\cup\{\infty\}$ is jointly measurable.
    \end{enumerate}
\end{proposition}

The requirement \eqref{eq:expected_action_finite} of assumption \ref{cond:cone} follows from
\begin{multline*}
    \E\sup_{s\in [0,1]}|\Amin(0,sx)| \le \E \sup_{s\in [0,1]}\int_0^{s}  \sqrt{g_{ux}(x,x)}du 
    \\
    \le |x|\int_0^1\E \|g_{ux}\|^{1/2}du = |x|\E\|g_{0}\|^{1/2} < \infty.
\end{multline*}

\subsection{Conditions \ref{cond:setup_diff}--\ref{cond:M_infty}}\label{sec:secondDerivRiemannianFPP}

Condition \ref{cond:setup_diff} is implied by \Cref{lem:XiProperties} and \ref{phiwDiffCondition}. 

We now verify \ref{cond:M_infty}. We fix $v\in \R^d\setminus\{0\}$ 
(and $H$) and, for $\delta>0$,  define $H(\delta)$
according to~\eqref{eq:H-delta}. 

We let $G^w$ be the Riemannian metric defined by $G^w_x(p,p) =  g^{w,v}_x(\Xi_{v\to w}p, \Xi_{v\to w} p)$. In particluar, 
\begin{equation}
    \label{eq:Gg}
    G^v_x = g_x,\quad x\in\R^d.
\end{equation}
The transformed action in 
\eqref{eq:Bwvgamma1} can be rewritten as
\begin{equation}
    \label{eq:Bwvgamma2}
    B(w,v,\gamma) = \int_0^t \sqrt{G^w_{\gamma_s} (\dot{\gamma}_s,\dot{\gamma}_s)} ds.     
\end{equation}

Recall that $\YField$ is the random function given in condition \ref{phiwDiffCondition}. We postpone the proof of the following lemma until \Cref{sec:proofs_of_lemmas}.

\begin{lemma}\label{lem:GDiffBound}
    For each $v\neq 0$, there is $c,C>0$ (depending on $v$) such that with probability one, for all $i=1,\dots,d$, $x\in \R^d$, and $w\in H(\delta)$ (where $\delta$ is defined in \ref{gBound}),
    \begin{align}\label{eq:GUpperBound}
        \|G^w_x\| + \|\partial_{w_i}G^w_x\| + \|\partial_{w_jw_i}G^w_x\| \le C \YField(x),
    \end{align}
    and, for all $p\in \R^d$,
    \begin{equation}\label{eq:GLowerBound}
        G^w_x(p,p) \ge c|p|^2.
    \end{equation}
\end{lemma}

\Cref{prop:measurableSelectionFPP} implies that there is $t>0$ and a path $\gamma^T=\gamma^T(v)\in \SP_{0,Tv,t}$ such that 
\begin{equation}
    B(v,v,\gamma^T(v)) =\Bmin^T(v)
\end{equation}
and
\begin{equation}\label{eq:parametrizationB}
    G^v_{\gamma_s^T(v)} (\dot{\gamma}^T_s(v),\dot{\gamma}^T_s(v)) = 1,\quad \forall s\in[0,t],
\end{equation}
where $t = A^T(v).$
We note that
\begin{align}\label{eq:formulaForPartialB}
    \partial_{w_i} B(w,v,\gamma^T) = \frac{1}{2}\int_0^t \frac{\partial_{w_i}G^w_{\gamma_s^T} (\dot{\gamma}_s^T,\dot{\gamma}_s^T)}{\sqrt{G^w_{\gamma_s^T} (\dot{\gamma}_s^T,\dot{\gamma}_s^T)}}ds
\end{align}
and 
\begin{multline}\label{eq:BSecondDerivative}
    \partial_{w_j}\partial_{w_i} B(w,v,\gamma^T) = \frac{1}{2}\int_0^t \frac{\partial_{w_j w_i} G^w_{\gamma_s^T} (\dot{\gamma}_s^T,\dot{\gamma}_s^T )}{\sqrt{G^w_{\gamma_s^T} (\dot{\gamma}_s^T,\dot{\gamma}_s^T)}}ds 
    \\ - \frac{1}{4}\int_0^t \frac{\partial_{w_j} G^w_{\gamma_s^T} (\dot{\gamma}_s^T,\dot{\gamma}_s^T )\partial_{w_i} G^w_{\gamma_s^T} (\dot{\gamma}_s^T,\dot{\gamma}_s^T )}{G^w_{\gamma_s^T} (\dot{\gamma}_s^T,\dot{\gamma}_s^T)^{3/2}}ds
\end{multline}
for $i,j=1,\dots, d$. 

First we use \eqref{eq:formulaForPartialB} to derive $\langle \nabla B(v,v,\gamma^T(v)), w\rangle$ for $w\in \R^d$, which will imply the formula in \eqref{eq:RiemannianFPPDerivative} once \Cref{th:main_riemannian} is established. Recall the definition $h_x^i(p;v)=\partial_{w_i}g^{w,v}_x(p,p)\Big|_{w=v}$ from \Cref{sec:RiemannianFPPeg}.
\begin{lemma}
    For $v,w\in \R^d$ and $x\in \R^d,$
    \begin{equation}
        \langle \nabla B(v,v,\gamma^*), w\rangle = \int_0^t\Big[\frac{1}{2}\sum_{i=1}^d h_{\gamma_s^*}^{i}(\dot{\gamma}_s^*;v) w_i + \frac{\langle v,\dot{\gamma}_s^*\rangle}{|v|^2} g_{\gamma_s^*}(w,\dot{\gamma}_s^*)\Big] ds,
    \end{equation}
    where $\gamma^* := \gamma^*(0,x)\in \mathcal{S}_{0,y,t}.$
\end{lemma}
\begin{proof}
    Note that \eqref{eq:formulaForPartialB} and \eqref{eq:gammaStarSpeed} imply 
    \begin{equation}\label{eq:BDerivStep1}
        \partial_{w_i} B(v,v,\gamma^*) = \frac{1}{2}\int_0^t \partial_{w_i}G^v_{\gamma_s^*}(\dot{\gamma}_s^*,\dot{\gamma}_s^*)ds.
    \end{equation}
    Also, 
    \begin{equation}
        \partial_{w_i}\Xi_{v\to w} p = \frac{\langle v,p\rangle}{|v|^2}e_i,
    \end{equation}
    where $e_i\in \R^d$ is the vector with $0$ in all coordinates $j\neq i$ and $1$ in coordinate $i$. Thus,
    \begin{align*}
        \partial_{w_i}G^w_{x}(p,p) & = \partial_{w_i}(g^{w,v}_x(\Xi_{v\to w} p,  \Xi_{v\to w} p)\\
        & = (\partial_{w_i} g^{w,v}_x)(\Xi_{v\to w} p,  \Xi_{v\to w} p) + 2 g_x^{w,v}\Big(\frac{\langle v,p\rangle}{|v|^2}e_i, p\Big).
    \end{align*}
    The formula \eqref{eq:BDerivStep1}, the identity $\Xi_{v\to v} = I,$ and the above imply that
    \begin{align*}
        \langle \nabla B(v,v,\gamma^*), w\rangle &= \int_0^t \Big[\frac{1}{2}\sum_{i=1}^d h_{\gamma^*_s}^{i}(\dot{\gamma}^*_s;v) w_i +\sum_{i=1}^d \frac{\langle v,\dot{\gamma}_s^*\rangle}{|v|^2}  g_{\gamma_s^*}(e_i,\dot{\gamma}_s^*) w_i\Big]ds\\
        & = \int_0^t \Big[\frac{1}{2}\sum_{i=1}^d h_{\gamma^*_s}^{i}(\dot{\gamma}^*_s;v) w_i +\frac{\langle v,\dot{\gamma}_s^*\rangle}{|v|^2}  g_{\gamma_s^*}(w,\dot{\gamma}_s^*) \Big]ds
    \end{align*}
    completing the proof.
\end{proof}

For a measurable set $S\subset \R^d$ and a path $\gamma$ let $\tau_S(\gamma) = \int_0^t \1_{\gamma_s\in S}ds.$ For $k\in \Z^d$, let 
\begin{equation}
    \label{eq:cubes}
I_k = k + [0,1)^d.
\end{equation}
We will also need the Euclidean length of the path $\gamma\in \SP_{*,*,t}$ defined as 
 \begin{equation}\label{eq:euclideanLength}
     \Length(\gamma) = \int_0^t |\dot{\gamma}_s|ds.
 \end{equation}

The following lemma is an application of the theory of greedy lattice animals (see, e.g., \cite{10.1214/aoap/1177005277}) to bounding actions of continuous paths. Its proof can be found in \Cref{sec:proofs_of_lemmas}.
\begin{lemma}\label{lem:abstractIntegralBound}
    Let $Z$ be a random function that satisfies the conditions of \ref{boundedY}
     for some $\beta>2d$. Let $(X_k)_{k\in \Z^d}$ be a collection satisfying the following conditions:
    \begin{enumerate}
        \item $X$ is stationary with respect to lattice shifts, meaning for every $a\in \Z^d,$ $(X_{x+a})_{x\in \Z^d}$ is equal in distribution to $(X_x)_{x\in \Z^d}$.
        \item $X$ has finite range dependence on the lattice.
        \item $\E[|X_0|^{\beta}] < \infty$ for some $\beta > 2d.$
    \end{enumerate}
    Let 
    \[\Gamma_X := \{\gamma\in \mathcal{S}_{0,\ast,\ast}\,:\,\forall k\in \Z^d,\,\tau_{I_k}(\gamma)\le X_k\}.\]
    Then, with probability one, 
    \begin{equation}
        \sup_{\gamma\in \Gamma_X}
        \frac{\int_0^{t(\gamma)} Z(\gamma_s)ds}{\Length(\gamma)+1} < \infty.
    \end{equation}
\end{lemma}

Using \Cref{lem:abstractIntegralBound} we can now prove the following proposition, which verifies \ref{cond:M_infty} since $\|\nabla_H^2 B\| \le \|\nabla^2 B\|$.
\begin{lemma}\label{lem:FPPSecondDerivativeBound}
   For $\delta$ given by Lemma~\ref{lem:GDiffBound}, with probability one,
    \[\limsup_{T\to \infty}\frac{1}{T}\sup_{w\in H(\delta)}\|\nabla^2 B(w,v,\gamma^T)\| < \infty.\]
\end{lemma}

\begin{proof}
    In this proof, the notation $C$ refers to a nonrandom positive number that may change line by line. By \Cref{lem:GDiffBound} and \eqref{eq:BSecondDerivative}, for all $w\in H(\delta)$ we have 
    \begin{align*}
        \|\nabla^2 B(w,v,\gamma^T)\|& \le C \int_0^t \frac{|\dot{\gamma}_s^T|^2 \YField(\gamma_s^T)}{|\dot{\gamma}^T_s|}ds + C \int_0^t \frac{|\dot{\gamma}_s^T|^4 \YField^2(\gamma_s^T)}{|\dot{\gamma}_s^T|^{3}}ds\\
        & \le C \sup_{s\in [0,t]}|\dot{\gamma}_s^T|\int_0^t |\max(\YField(\gamma_s^T),1)|^2ds.
    \end{align*}
    Since \eqref{eq:parametrizationB} holds and $G^v$ is uniformly positive definite, $|\dot{\gamma}_s^T| \le c^{-1}G^v_{\gamma_s^T}(\dot{\gamma}_s^T, \dot{\gamma}_s^T) = c^{-1}$ for all $s\in [0,t].$ Thus, there is $C>0$ such that for all $T>0$, 
    \begin{equation}\label{eg:nabla2BBound}
        \nabla^2 B(w,v,\gamma^T)\le C \int_0^t \max(\YField(\gamma_s^T),1)^2ds.
    \end{equation}

    We will now bound $\int_0^t \max(\YField(\gamma_s^T),1)^2ds$ using \Cref{lem:abstractIntegralBound}. Clearly since $\YField$ follows \ref{boundedY} with $\beta> 4d$, $\max(\YField(\gamma_s^T),1)^2$ will obey \ref{boundedY} with $\beta>2d.$ The collection $(X_k)_{k\in \Z^d}$ in \Cref{lem:abstractIntegralBound} can be chosen to be 
    \begin{equation}\label{eq:XkDef}
        X_k = C\sup_{z\in I_k}|\YField(z)|^{1/2}
    \end{equation}
    for a sufficiently large constant $C$. Since $\gamma^T$ is a geodesic with unit speed, we have for all $k\in \Z^d$,
    \begin{align}\label{eq:psiTinGammaX}
        \tau_{I_k}(\gamma^T) &= \int_0^t \1_{\gamma_s^T\in I_k}ds\\
        & = \int_0^t \1_{\gamma_s^T\in I_k} \sqrt{G^v_{\gamma_s^T}(\dot{\gamma}_s^T,\dot{\gamma}_s^T)} ds\notag\\
        & \le X_k\notag
    \end{align}
    if $C$ in \eqref{eq:XkDef} is sufficiently large.
    Indeed, if we let $s_k$ and $r_k$ be, respectively, the first and last times $s$ such that $\gamma_s^T\in \overline{I_k}$, then
    \[\int_{s_k}^{r_k} \sqrt{G^v_{\gamma_s^T}(\dot{\gamma}_s^T,\dot{\gamma}_s^T)} ds = \inf_{\gamma:\gamma_{s_k}^T\to \gamma_{r_k}^T}\int_0^t \sqrt{G^v_{\gamma_s} (\dot{\gamma}_s,\dot{\gamma}_s )}ds.\]
    The right-hand side is bounded above by $\sup_{x,y\in I_k}|x-y|\sup_{z\in I_k}\|G^v_z\|^{1/2}$ because the path $\gamma$ that is linear between $\gamma^T_{s_k}$ and $\gamma_{t_k}^T$ is admissible. We can then conclude  by \Cref{lem:GDiffBound} that \eqref{eq:psiTinGammaX} holds with $X_k$ defined in \eqref{eq:XkDef} for some $C$.

    Inequality \eqref{eq:psiTinGammaX} implies that $\gamma^T$ is in the set $\Gamma_X$ defined in \Cref{lem:abstractIntegralBound}. So, \Cref{lem:abstractIntegralBound} implies that 
    \begin{equation}\label{eq:supLengthBound}
        \sup_{T > 0}\frac{\int_0^t |\max(\YField(\gamma_s^T),1)|^2ds}{\Length(\gamma^T)+1} <\infty.
    \end{equation}
    Because $\sup_{s\in [0,t]}|\dot{\gamma}^T_s| \le c^{-1}$, 
    \begin{equation}
        \Length(\gamma^T)  = \int_0^t |\dot{\gamma}_s^T| ds \le c^{-1} t = c^{-1}\Amin(\gamma^T) = c^{-1}\Amin^T(v).
    \end{equation}
    Since $\Amin(0,x) \le \int_0^1 \sqrt{g_{sx}(x,x)}ds \le |x|\int_0^1 \|g_{sx}\|^{1/2}ds$,  ergodicity of $g$ with respect to spatial shifts implies that $\limsup_{T\to \infty}\frac{\Amin^T(v)}{T} < \infty$ with probability one. Thus,
    \begin{equation}\label{eq:lengthLinearGrowth}
        \limsup_{T\to \infty}\frac{\Length(\gamma^T)}{T}< \infty
    \end{equation}
    with probability one. Displays \eqref{eq:lengthLinearGrowth}, \eqref{eq:supLengthBound}, and \eqref{eg:nabla2BBound} imply \Cref{lem:FPPSecondDerivativeBound}.
\end{proof}

\subsection{Condition \ref{cond:for-finiteness}}
\label{sec:checking-cond-for-finiteness}
Since $\Amin(0,x) = \Amin(x,0)$ for all $x\in \R^d$ and all $\omega\in \Omega$, it suffices to prove \eqref{eq:linearGrowthProb}.  The latter is implied by the following lemma: 
\begin{lemma}\label{lem:linearGrowthRiemannianFPP}
    With probability one,
    \begin{equation}\label{eq:ActionLinearGrowthRiemannian}
        \sup_{x\in \R^d}\frac{\Amin(0,x)}{|x|+1} < \infty.
    \end{equation}
\end{lemma}
\begin{proof} 
    We will apply \Cref{prop:latticeAnimalsFiniteRange}. For $x\in \R^d$ let $\gamma^x\in \SP_{0,x,|x|}$ denote the path $\gamma^x_s = \frac{sx}{|x|}.$ Let $K(x)$ denote those $k\in \Z^d$ such that $\gamma^x\cap I_k\ne \emptyset$. Define $X_k = \sup_{z\in I_k}|\YField(z)|^{1/2}.$ The collection $(X_k)_{k\in \Z^d}$ satisfies the conditions of \Cref{prop:latticeAnimalsFiniteRange} due to \ref{boundedY}. Then, 
    \begin{multline*}
        \Amin(0,x) \le \int_0^{|x|}\sqrt{g_{\gamma^x_s}(\dot{\gamma}^x_s,\dot{\gamma}^x_s)}ds
         \le \int_0^{|x|} \|g_{\frac{sx}{|x|}}\|^{1/2}ds
         \\
        \le \sum_{k\in K(x)}\int_0^{|x|}X_k \1_{\frac{sx}{|x|}\in I_k}ds
         \le \sqrt{d}\sum_{k\in K(x)}X_k.
    \end{multline*}
    In the last line we use the fact that $\int_0^{|x|}\1_{\frac{sx}{|x|}\in I_k}ds \le \sqrt{d}$ for all $x\in \R^d$ and $k\in \Z^d.$ Also, there is $C>0$ such that for all $x\in \R^d,$ $|K(x)|\le C |x|+C$, where $|K(x)|$ denote the number of elements in the finite set $K(x).$ Thus,
    \[\sup_{x\in \R^d}\frac{\Amin(0,x)}{|x|+1}\le \sup_{x\in \R^d}\frac{C\sqrt{d}}{|K(x)|}\sum_{k\in K(x)}X_k.\]
    Note also that $K(x)$ is a $\ast$-connected set as defined 
    in~\Cref{sec:proofs_of_lemmas}. \Cref{prop:latticeAnimalsFiniteRange} implies that the right-hand side of the above is finite with probability one, and so our proof of \Cref{lem:linearGrowthRiemannianFPP} is complete.
\end{proof}

\subsection{Proof of \Cref{prop:measurableSelectionFPP}}\label{sec:checkingConditionsFPP}

We will appeal to an abstract measurable selection lemma, which we first state below. See \Cref{sec:proofs_of_lemmas} for its proof. Recall if $Y$ is a Banach space then the weak topology on $Y$ refers to the topology induced by maps $f:Y\to \R$ in the dual space $Y^*.$ A ball $B$ in a Banach space $Y$ refers to a set of the form $\{y\in Y\,:\,\|y-y_0\|_Y<r\}$ for $y_0\in Y$ and $r>0.$ The notation $\overline{B}$ refers to the closure of the ball.

\begin{lemma}\label{lem:measurableSelectionLem}
    Let $(X,\mathcal{F})$ be a measurable space and $(Y, \|\cdot\|_Y)$ be a separable Banach space. Suppose $F:X\times Y\to \R\cup\{\infty\}$ satisfies the following.
    \begin{enumerate}
        \item\label{seqCompact} For every $x\in X$ and $R\in \R$ the set $F^{-1}(x;R):= \{y\in Y\,:\,F(x,y) \le R\}$ is weakly compact and, additionally, the set $\{y\in Y\,:\,F(x,y) < \infty\}$ is non-empty.
        \item\label{countableSubset} There is a countable subset $\mathcal{G}\subset X$ such that for all balls $B\subset Y$, $x\in X$, $R<\infty$, and $\epsilon > 0$ there is $x' = x'(B,x,R,\epsilon)\in \mathcal{G}$ such that 
        \[|F(x,y) - F(x',y)| < \epsilon\] 
        for all $y\in \overline{B}$ satisfying either $F(x,y) \le R$ or $F(x',y) \le R.$
        \item\label{measurability} For every $y\in Y$ the map $x\mapsto F(x,y)$ is measurable. 
    \end{enumerate}
    Then, there is a measurable function $f:(X,\mathcal{F})\to (Y,\|\cdot\|)$ satisfying
    \begin{equation}\label{eq:measurableSelection}
        F(x,f(x)) = \inf_{y\in Y}F(x,y)
    \end{equation}
    for all $x\in X.$
\end{lemma}

\begin{proof}[Proof of \Cref{prop:measurableSelectionFPP}]
    Let $X_0$ be the space of those $g\in C^2_{loc}(\R^d;\mathcal{M}^d_+)$ that satisfy $g_x(p,p)\ge \lambda |p|^2$ for all $x,p\in \R^d.$ We will apply \Cref{lem:measurableSelectionLem} with $X = \R^d\times \R^d\times X_0$ and $Y = \{h\in L^1([0,1];\R^d)\,:\,\int h = 0\}$ (equipped with the $L^1$ norm). 
    Given $(x,y)\in \R^d\times \R^d$ and $h\in Y$ we let $\gamma = \gamma[x,y,h]\in \SP_{x,y,1}$ be given by $\gamma_s = (1-s) x + sy + \int_0^s h(r)dr$. 
     Note that the map $(x,y,h)\mapsto \gamma[x,y,h]$ is continuous and $\dot{\gamma}_s[x,y,h] = y-x + h_s$.

     Let us define the ``energy functional'' $F:X\times Y\to \R \cup\{\infty\}$  by
    \[F(x,y,g,h) = \int_0^1 g_{\gamma_s}(\dot{\gamma}_s,\dot{\gamma}_s)ds,\]
    where $\gamma = \gamma[x,y,h]$.
    
    We claim that to prove part \eqref{it:min-seclection} of \Cref{prop:measurableSelectionFPP}, it suffices to prove the existence of a measurable selection $f:X\to Y$ satisfying 
    \begin{equation}\label{eq:defOff}
        F(x,y,g,f(x,y,g)) = \inf_{h\in Y}F(x,y,g,h).
    \end{equation}
    To see this, we first let $\tilde{\gamma}$ denote the mapping $(x,y,g)\mapsto {\gamma}[x,y,f(x,y,g)]$, which is also measurable if $f$ is. Suppose $(x,y,g)\in X$. 
    By Lemma 2.3 in Chapter 9 of \cite{alma990020579810107876}, the energy minimizing path $\tilde{\gamma}=\tilde{\gamma}(x,y,g)$ has constant speed: $g_{\tilde{\gamma}}(\dot{\tilde{\gamma}},\dot{\tilde{\gamma}}) \equiv c$  for some $c>0,$ and the path $\gamma^*=\gamma^*(x,y,g) \in \SP$ defined by $\gamma^*_s = \tilde{\gamma}_{c^{-1/2}s}$ minimizes
    \[\int_0^t g_{\psi_s}(\dot{\psi}_s,\dot{\psi}_s)ds,\quad t>0,\,\psi\in \SP_{x,y,t}.\]
    Additionally, $\gamma^*$ has unit Riemannian speed: $g_{\gamma^*}(\dot{\gamma}^*,\dot{\gamma}^*) \equiv 1$. Due to measurability of the map $\omega\mapsto g_{\cdot,\omega}$ and measurability of the reparametrization operation $\tilde \gamma\mapsto \gamma^*$, the mapping $(x,y,\omega)\mapsto \gamma^*(x,y,\omega)$ is measurable and satisfies~\eqref{eq:gammaXYDef}. This proves our claim that it suffices to find $f$ satisfying \eqref{eq:defOff}.
    
    To prove existence of a measurable selection $f:X\to Y$ satisfying \eqref{eq:defOff},we will use \Cref{lem:measurableSelectionLem}.
 
    First we verify condition \eqref{seqCompact} of \Cref{lem:measurableSelectionLem}. Because the map $p\mapsto g_x(p,p)$ is convex and nonnegative, Corollary 3.24 of \cite{dacorogna2007direct} 
      implies that the map $\gamma\mapsto \int_0^1 g_{\gamma_s}(\dot{\gamma}_s,\dot{\gamma}_s)ds$ is weakly lower semicontinuous in $W^{1,1}([0,1];\R^d).$ Additionally, by condition \ref{uniformPositiveDefinite}, the set $\{h\in Y\,:\,F(x,y,g,h)\le R\}$
    is bounded in $L^2$ norm and thus is contained in a weakly compact set.  
    It follows that $\{h\in Y\,:\, F(x,y,g,h)\le R\}$ is weakly compact for every $R>0.$ 
 
    The set $\{h\in Y\,:\, F(x,y,g,h)<\infty\}$ is nonempty since it contains the function $h\equiv 0$. This completes the proof of part~\eqref{it:min-seclection}.

    Now we establish condition \ref{countableSubset} of \Cref{lem:measurableSelectionLem}. Since the space $X_0$ is separable, we can find a countable subset $\mathcal{G}_0$ satisfying
    \begin{equation}\label{eq:GNullDeff}
        \inf_{g\in \mathcal{G}_0}\sup_{x\in K}\|g' - g\|_{C^2,x} = 0
    \end{equation}
    for all $g'\in X_0$ and all compact sets $K\subset \R^d.$ Then, let the set $\mathcal{G}$ in \Cref{lem:measurableSelectionLem} be $\Q^d\times \Q^d\times \mathcal{G}_0.$

    Fix $(x,y)\in \R^d\times \R^d$, $g\in X_0,$ and $B$ an open ball in the Banach space $Y$. There is a constant $C$ depending on $x,y,B$ such that for all $h\in B$, $\Length(\gamma[x,y,h])\le |x-y| + \|h\|_{L^1} \le C.$  So, there is a compact set $K(x,y,B)\subset \R^d$ such that 
    \begin{equation}\label{eq:defofKxyB}
        \gamma_s[x,y,h]\in K(x,y,B),\quad \forall h\in \overline{B},\forall s\in [0,1].
    \end{equation}
    Also, \ref{uniformPositiveDefinite} implies that for all $(x,y,g)\in X$ and $h\in Y$, 
    \begin{equation}\label{eg:hL2Bound}
        \|h\|_{L^2} \le |x-y| + \|\dot{\gamma}\|_{L^2} \le |x-y| + \lambda^{-1/2}F(x,y,g,h)^{1/2}.
    \end{equation}

    In the below computation we let $h\in B,$ $(x,y,g),(x',y',g')\in X$, $\gamma = \gamma[x,y,h]$ and $\gamma' = \gamma[x',y',h]$. Then,
    \begin{align}\label{eq:FandFPrimeBound}
        |F(x',y',g',&h) - F(x,y,g,h)| \le \int_0^1 |g_{\gamma'_s}'(\dot{\gamma}'_s,\dot{\gamma}'_s) - g_{\gamma_s}(\dot{\gamma}_s,\dot{\gamma}_s)|ds\notag\\
        & \le \int_0^1 |g_{\gamma'_s}'(\dot{\gamma}'_s-\dot{\gamma}_s,\dot{\gamma}'_s)| +|g'_{\gamma'_s}(\dot{\gamma}_s, \dot{\gamma}_s' - \dot{\gamma}_s)| +  |(g_{\gamma_s'}' - g_{\gamma_s})(\dot{\gamma}_s,\dot{\gamma}_s)|ds.
    \end{align}
    Let $R<\infty$ and suppose either $F(x',y',g',h) \le R$ or $F(x,y,g,h) \le R$. By \eqref{eg:hL2Bound}, $\|h\|_{L^2} \le \max(|x-y|,|x'-y'|) + \lambda^{-1/2}R^{1/2}$ and so $\|\dot{\gamma}\|_{L^2}$ and $\|\dot{\gamma}'\|_{L^2}$ are both bounded by $a := 2 \max(|x-y|,|x'-y'|) + \lambda^{-1/2}R^{1/2}.$
    
    Note that $\dot{\gamma}'_s - \dot{\gamma}_s = (x-x') + (y'-y)$. We can bound the first and second terms in the right-hand side of \eqref{eq:FandFPrimeBound} by 
    \begin{align*}
        \sup_{s\in [0,1]}\|g'_{\gamma'_s}\| \|\dot{\gamma}'-\dot{\gamma}\|_{L^2}\max(\|\dot{\gamma}'\|_{L^2}, \|\dot{\gamma}\|_{L^2}) \le a(|x-x'|+|y-y'|)\sup_{x\in K}\|g'_{x}\|,
    \end{align*}
    where $K$ is defined as the union of $K(x,y,B)$ and $K(x',y',B)$ from \eqref{eq:defofKxyB}.

    We can bound the third term of \eqref{eq:FandFPrimeBound} by 
    \begin{align*}
        \sup_{s\in [0,1]}\|g'_{\gamma_s'} &- g_{\gamma_s}\|\|\dot{\gamma}\|_{L^2}^2\le \Big(\sup_{x\in K}\|g'_x - g_x\|+ \sup_{x\in K}\|g\|_{C^1,x}\|\dot{\gamma}' - \dot{\gamma}\|_{L^\infty}\Big)\|\dot{\gamma}\|_{L^2}^2\\
        &  \le a^2 \Big(\sup_{x\in K}\|g'_x-g_x\| + \sup_{x\in K}\|g\|_{C^1,x}(|x-x'|+|y-y'|)\Big).
    \end{align*}
    Fix now $R>0,$ $x,y\in \R^d$, an open ball $B\subset Y,$ and $\epsilon > 0.$ By choosing $x',y'\in \Q^d$ sufficiently close to $x$ and $y$, respectively, and $g'\in \mathcal{G}_0$ such that $\sup_{x\in K}\|g'-g\|_{C^2,x}$ is sufficiently small, we can guarantee that \eqref{eq:FandFPrimeBound} is less than $\epsilon$ for all $y\in \overline{B}$ satisfying either $F(x',y',g',h)\le R$ or $F(x,y,g,h)\le R.$ This implies \ref{countableSubset} of \Cref{lem:measurableSelectionLem}. 
    
    Condition \ref{measurability} is satisfied because the mapping $(x,y,g)\mapsto \int_0^1 g_{\gamma_s}(\dot{\gamma}_s,\dot{\gamma}_s)ds$ is continuous for every $\gamma \in \SP_{x,y,1}$ and so the proof is complete.
\end{proof}

\section{Proofs of Lemmas from \Cref{sec:proof-of-main-riemannian}}
\label{sec:proofs_of_lemmas}

\begin{proof}[Proof of \Cref{lem:GDiffBound}]
    For every $i,j =1,\dots, d,$ and $y\in \R^d$
\begin{equation}
    \partial_{w_i}\Xi_{v\to w} y = \frac{\langle v,y\rangle}{|v|^2 }e_i,\quad \partial_{w_i,w_j}\Xi_{v\to w} y = 0.
\end{equation}
In particular, $\|\partial_{w_i}\Xi_{v\to w}\|$ and $\|\partial_{w_i,w_j}\Xi_{v\to w}\|$ are bounded by some constant depending only on $v$. Additionally, $\| \Xi_{v\to w}\|$ itself is bounded uniformly for $w$ in a neighborhood of $v$.

Using the product rule for matrices, we can derive 
\begin{align}\label{eq:GDiffBound1}
    \|\partial_{w_i}G^w_x\| &\le 2 \|\Xi_{v\to w}\| \|\partial_{w_i}\Xi_{v\to w}\| \|g^{w,v}_x\| + \|\Xi_{v\to w}\|^2 \|\partial_{w_i}g^{w,v}_x\|\notag \\
    & \le C_3 \max(\|g_x^{w,v}\|,\|\partial_{w_i}g^{w,v}_x\|)\notag\\
    & \le C_4 \YField(x)
\end{align}
for $|w-v| \le \delta$ for $\delta$ as in \ref{gBound}. Similarly, 
\begin{align}\label{eq:GDiffBound2}
    \|\partial_{w_jw_i}G^w_x\| &\le C_5 \max(\|g^{w,v}_x\|,\|\partial_{w_i}g^{w,v}_x\|, \|\partial_{w_jw_i}g^{w,v}_x\|)\notag\\
    &\le C_6 \YField(x).
\end{align}
Displays \eqref{eq:GDiffBound1} and \eqref{eq:GDiffBound2} complete our proof of \eqref{eq:GUpperBound}.

To see \eqref{eq:GLowerBound} simply note that $g^{w,v}$ itself is uniformly positive definite, and $\Xi_{v\to w}$ is invertible for every $w\in v+H$ by \Cref{lem:XiProperties}. So, for every $w\in v+H$ there is $c(w)>0$ such that for all $x,p\in \R^d,$
\[G^w_x(p,p) = g^{w,v}_x(\Xi_{v\to w}p,\Xi_{v\to w}p) \ge \lambda |\Xi_{v\to w}p|^2 \ge \lambda c(w)|p|^2.\]
The constant $c(w)$ can be chosen as the square of the smallest eigenvalue of $\Xi_{v\to w}$. Because the map $w\mapsto \Xi_{v\to w}$ is continuous, there is a $c >0$ such that $G^w_x(p,p) \ge c |p|$ for all $x,p\in \R^d$ and $w\in v+H$ satisfying $|v-w|\le 1,$ and so the lemma is proved.
\end{proof}

To prove \Cref{lem:abstractIntegralBound} we use a greedy lattice animal estimate extending implied by the results in \cite{10.1214/aoap/1177005277} and \cite{MR1884923}. We can consider $\Z^d$ as a graph where $x$ is connected to $y$ whenever 
\[\max_{i=1,\dots, d}|x_i - y_i| = 1.\]
We say that $A\subset \Z^d$ is $\ast$-connected whenever it is a connected component of the aforementioned graph. We let $\mathcal{C}(n)$ denote the set of all $\ast$-connected subsets of $\Z^d$ with $n$ elements containing the origin.

\begin{proposition}\label{prop:latticeAnimalsFiniteRange}
    Let $(X_k)_{k\in\Z^d}$ be a collection of nonnegative random variables obeying the following conditions:
    \begin{enumerate}
        \item $(X_k)_{k\in\Z^d}$ is stationary with respect to lattice shifts, meaning for every $a\in \Z^d,$ $(X_{k+a})_{k\in \Z^d}$ is equal in distribution to $(X_k)_{k\in \Z^d}$.
        \item $(X_k)_{k\in\Z^d}$ has finite range dependence on the lattice.
        \item\label{item:cdfBound} $\int_0^\infty(1-F(x))^{1/d}dx < \infty,$ where $F$ is the cdf of $X_0$.
    \end{enumerate}
    Then,
    \begin{equation}\label{eq:latticeAnimalsFinite}
        \sup_{n\in \N}\frac{1}{n}\max_{A\in \mathcal{C}(n)}\sum_{k\in A}X_k  <\infty.
    \end{equation}
\end{proposition}
\begin{remark}
    As remarked in \cite{MR1884923}, if $\E |X_0|^\beta < \infty$ for some $\beta > d,$ then condition \ref{item:cdfBound} is satisfied: $\int_0^\infty(1-F(x))^{1/d}dx < \infty$.
\end{remark}

\begin{proof}
    We will use Theorem 1.1 in~\cite{MR1884923}. There are two differences between our setting and that of~\cite{MR1884923}. First, in~\cite{MR1884923}, two nodes $x,y\in\Z^d$ are connected by an edge if 
    \[\sum_{i=1}^d |x_i-y_i| = 1.\]
    We call $B\subset \Z^d$ $\ell^1$-connected if it is connected with respect to this $\ell^1$ graph structure. We denote by $\mathcal{C}_1(n)$ the set of all $\ell^1$-connected components of $\Z^d$ containing the origin of size $n$. If $A$ is a $\ast$-connected subset of $\Z^d$ of size $n$, then there is an $\ell^1$-connected subset of $\Z^d$ of size at most $2^d n$ such that $A\subset B$. For instance, the set $B$ can be constructed by adding all $\ell^1$-nearest neighbors of elements in $A$. Since $X_k \ge 0$, this argument shows that 
    \[\max_{A\in \mathcal{C}(n)}\sum_{k\in A}X_k  \le \max_{A\in \mathcal{C}_1(2dn)}\sum_{k\in A}X_k,\]
    so  it suffices to prove \eqref{eq:latticeAnimalsFinite} for the case where $\mathcal{C}(n)$ is replaced by $\mathcal{C}_1(n)$ defined via $\ell^1$-connectedness,  the graph structure considered in \cite{MR1884923}.

    The second difference with Theorem 1.1 in \cite{MR1884923} is that in our case the random variables $(X_x)_{x\in \Z^d}$ are not independent but rather have finite range of dependence. However, we can reduce the problem to the i.i.d. case. Let $K\in \N$ be an upper bound on the dependence range of $(X_k)_{k\in \Z^d}$ and let $M_K = [0,K)^d\cap \Z^d.$

    For $k\in M_K$, we can regard 
    \[E_k := k + K \Z^d = \{k + K x\,:\,x\in \Z^d\}\]
    as a graph isomorphic to $\Z^d$ with $\ell^1$ nearest neighbor edges. Additionally, 
    \[\bigcup_{k\in M_K} E_k = \Z^d.\]
    The i.i.d.\ family $(X_{x})_{x\in E_k}$ satisfies the requirements of \cite{MR1884923}. Let $\mathcal{C}_1(n,k)$ denote the set of $\ell^1$-connected subsets of $E_k$ of size at most $n$ containing $k.$ We claim that for each $k\in M_K$ and $A\subset \mathcal{C}_1(n)$ there exists a set $F_k(A) \in \mathcal{C}_1(n,k)$ satisfying
    \begin{equation}\label{eq:Fproperty}
        A\cap E_k \subset F_{k}(A),\quad\forall A\in \mathcal{C}_1(n).
    \end{equation}
    If $x \in \Z^d$, then we can write $x = k + Ky$ for some (unique) $k\in M_K$ and $y\in \Z^d$. Let $R(x)$ denote the set $M_K + Ky$. Define the map $F_{k}$ in the following way:
    \begin{equation}\label{eq:FkDef}
        F_{k}(A) = \{x\in E_k\,:\, R(x)\cap A\neq \emptyset\}.
    \end{equation}
    Now suppose $A\in \mathcal{C}_1(n)$. The fact that
    \[\#\{x\in E_k \,:\,R(x)\cap A\}  \le \sum_{x\in E_k}|A\cap R(x)| = |A|\]
    implies that $|F_k(A)| \le n.$ Now let $x,x'\in F_k(A)$. Then, there are $z\in A\cap R(x)$ and $z'\in A\cap R(x')$ and an $\ell^1$-conected path $(w_0,\dots, w_m)$ in $A$ connecting $z$ to $z'.$ If $w'_i\subset E_k$ for $i=1,\dots,m$ are such that $w_i\in R(w'_i),$ then $(w'_i)_i$ is an $\ell^1$-connected path in $E_k$ and $w'_i\in F_k(A)$ for each $i$. Additionally, $w_0' = x$ and $w_{m}' = x'$. It follows that $x$ and $x'$ are connected by a path (considered as a subset of the graph $E_k$) in $F_k(A)\cap E_k$ and so the set $F_k(A)$ is connected as a subset of $E_k.$ Thus, $F_k(A) \in \mathcal{C}_1(n,k).$
    
    We have
    \begin{align*}
        \max_{A\in \mathcal{C}_1(n)}\sum_{x\in A}X_x &= \max_{A\in \mathcal{C}_1(n)}\sum_{k\in M_K}\sum_{x\in A\cap E_k}X_x\\
        & \le \max_{A\in \mathcal{C}_1(n)}\sum_{k\in M_K}\sum_{x\in F_{k}(A)}X_x
        \le \sum_{k\in M_K}\max_{B\in \mathcal{C}_1(n,k)}\sum_{x\in B}X_x.
    \end{align*}
    Theorem 1.1 in \cite{MR1884923} directly implies that 
    \[\sup_{n\in \N}\frac{1}{n}\max_{B\in \mathcal{C}_1(n,k)}\sum_{x\in B}X_x < \infty\]
    for each $k\in M_K$ and so \Cref{prop:latticeAnimalsFiniteRange} follows.
\end{proof}

\begin{proof}[Proof of \Cref{lem:abstractIntegralBound}]
    We will discretize our path and use \Cref{prop:latticeAnimalsFiniteRange}. Let 
    \[Z_k =  \sup_{x\in I_k} Z(x),\]  
    \[\chi_k = \begin{cases}
        1, & \exists s\text{ s.t. }\gamma_s\in I_k,\\
        0, & \text{otherwise},
    \end{cases}\]
    and $\chi(\gamma) = \{k\in \Z^d\,:\,\chi_k = 1\}$ for $\gamma\in \SP.$ If $\gamma \in \Gamma_X,$ then
    \begin{align}\label{eq:ZgammaUpperBound}
        \int_0^t Z(\gamma)ds \le \sum_{k\in \Z^d}Z_k \int_0^t \1_{\gamma_s\in I_k}ds & = \sum_{k\in \Z^d}Z_k \tau_{I_k}(\gamma)\\ 
& \le \sum_{k\in \Z^d}Z_k X_k\chi_k \notag
         \le \sqrt{\sum_{k\in \chi(\gamma)} Z_k^2  \sum_{k\in \chi(\gamma)} X_k^2}.\notag
    \end{align}
    Also, the collections $(Z_k^2)_{k\in \Z^d}$ and $(X_k^2)_{k\in \Z^d}$ both satisfy the conditions of \Cref{prop:latticeAnimalsFiniteRange}. Also, note that by continuity of $\gamma$, $\chi(\gamma)$ is a $\ast$-connected set in $\Z^d$ that contains the origin. In particular, $\chi(\gamma) \in \mathcal{C}(|\chi(\gamma)|)$. \Cref{prop:latticeAnimalsFiniteRange} implies that almost surely
    \begin{equation}
        \sup_{n\in \N}\max_{S\in  \mathcal{C}(n)} \frac{1}{n^2}\sum_{k\in S} Z_k^2  \sum_{k\in S} X_k^2 < \infty
    \end{equation}
    So, almost surely
    \begin{equation}\label{eq:sumsUpperBound3}
        \sup_{\gamma\in \mathcal{S}_{0,\ast,\ast}} \frac{1}{|\chi(\gamma)|^2}\sum_{k\in \chi(\gamma)} Z_k^2  \sum_{k\in \chi(\gamma)} X_k^2 < \infty.
    \end{equation}
    Following the argument of Lemma 4.5 in \cite{bakhtin2023differentiability}, one can show that there is $C>0$ such that for all paths $\gamma \in \mathcal{S}_{0,\ast,\ast}$,
    \begin{equation}\label{eq:discretizationUpperBound}
        |\chi(\gamma)| \le C\Length(\gamma) + C.
    \end{equation}
    The claim then follows by combining \eqref{eq:ZgammaUpperBound}, \eqref{eq:sumsUpperBound3}, and \eqref{eq:discretizationUpperBound}.
\end{proof}

\begin{proof}[Proof of \Cref{lem:measurableSelectionLem}]
    First, note that the Eberlein–Šmulian theorem (see Theorem 13.1 in Chapter V of \cite{alma991032421649703276}) 
     implies that weak compactness is equivalent to weak sequential compactness in a Banach space. Thus, condition \ref{seqCompact} of \Cref{lem:measurableSelectionLem} implies that $F(x,\cdot)$ is weakly sequentially lower semicontinuous for all $x\in X$.

    Consider the set-valued function 
    \begin{align*}
        \Psi:X&\to \mathcal{P}(Y)\\\
        x&\mapsto \{y\in Y\,:\,F(x,y) = \inf_{y'\in Y}F(x,y')\},
    \end{align*}
    where $\mathcal{P}(Y)$ is the power set of $Y$. We will apply the Kuratowski--Ryll-Nardzewski (KRL) Selection Theorem (see Theorem 18.13 in \cite{Aliprantis:MR2378491}). If the conditions of this theorem are met, then there is a measurable map $f:X\to Y$ such that $f(x) \in \Psi(x)$ for all $x\in X$, which is equivalent to \eqref{eq:measurableSelection}. First note that as a separable Banach space $Y$ is also a Polish space, one condition of the KRL theorem. We must additionally verify that the map $\Psi$ takes values in closed, nonempty sets and satisfies a set valued measurability condition known as being weakly measurable. 

    First we verify that $\Psi(x)$ is nonempty for all $x\in X.$ Let $I(x) = \inf_y F(x,y) < \infty.$ Take a sequence $(y_n)_{n\in \N}$ such that $F(x,y_n)\to I(x)$. The sequence $(y_n)_{n\in\N}$ has a weakly convergent subsequence to a $y^*\in Y$ because the set $\{y\,:\, F(x,y) \le I(x) + 1\}$ is weakly sequentially compact. Since $F$ is weakly sequentially lower semicontinuous, $F(x,y^*) = I(x)$, and so in fact $y^*\in \Psi(x)$.

    The fact that $\Psi$ takes value in closed sets follows directly from weak sequential lower semicontinuity of $F$. Indeed, if $(y_n)_{n\in \N}$ is a sequence such that $y_n \in \Psi(x)$ and $y_n\to y$, then $F(x,y)\le I(x)$, which implies $y\in \Psi(x)$.

    Now we prove that the map $\Psi$ is weakly measurable. Weak measurability means that for every open set $U\subset Y,$ the set 
    \[U_{\Psi^{-1}} := \{x\in X\,:\,\Psi(x)\cap U\neq 0\}\]
    is measurable in $X$. Let $\mathcal{A}\subset \mathcal{P}(Y)$ denote a countable basis of open balls in $Y$. For $x\in \mathcal{G}$ (where $\mathcal{G}$ is as in condition \ref{countableSubset} of the lemma) and $B\in \mathcal{A}$, define 
    \[I_{x,B}  := \inf \{F(x,y)\,:\,y\in B\}.\]
    If $I_{x,B} < \infty$, then there exists a $y_{x,B}\in \overline{B}$ satisfying $F(x,y_{x,B}) = I_{x,B}.$ Indeed, $\overline{B}$ is convex and strongly closed and thus weakly closed. It follows that the set $\{y'\in \overline{B}\,:\,F(x,y')\le I_{x,B}+1\}$ is weakly compact and $F(x,\cdot)$ is weakly lower semicontinuous, and so the existence of $y_{x,B}$ follows as soon as $I_{x,B} < \infty.$ For $B\in \mathcal{A}$, let $\mathcal{G}_B$ be those $x\in \mathcal{G}$ such that $I_{x,B} < \infty.$ For $U\subset Y$ open, let $\mathcal{A}_U$ denote those sets $B\in \mathcal{A}$ such that $\overline{B}\subset U.$ We claim that 
    \begin{equation}\label{eq:UpsiMeasurable}
        U_{\Psi^{-1}} = \bigcup_{B\in \mathcal{A}_U}\bigcap_{k=1}^\infty\bigcup_{x'\in \mathcal{G}_B}\bigcap_{y'\in Y_{\mathcal{G},\mathcal{A}}} C(k,y_{x',B}, y')
    \end{equation}
    where 
    \begin{equation}\label{eq:CkyyDef}
        C(k,y,y') := \{x\in X\,:\, F(x,y)\le F(x,y') + 1/k\}.
    \end{equation}
    For every $y\in Y$ the map $x\mapsto F(x,y)$ is measurable, and so $C(k,y,y')$ is measurable for each $k\in\N$ and $y,y'\in Y.$ It follows that if \eqref{eq:UpsiMeasurable} holds, then $U_{\Psi^{-1}}$ is measurable, and we can conclude that a measurable selection exists. We will now prove that the equality \eqref{eq:UpsiMeasurable} holds.

    First we will prove the forward inclusion of \eqref{eq:UpsiMeasurable}. Let $x$ be in $U_{\Psi^{-1}}$. By definition there is $y^*\in U$ such that $F(x,y^*) = I(x)$. Since $U$ is open, we can find $B\in \mathcal{A}_U$ such that $y^*\in B$. By condition \ref{countableSubset}, for all $k\in \N$ there exists $x_k\in \mathcal{G}$ such that for all $y\in \overline{B}$ satisfying either $F(x,y) \le I(x) + 1$ or $F(x_k,y) \le I(x) + 1,$
    \begin{equation}\label{eq:xkDef}
        |F(x_k,y) - F(x,y)| < \frac{1}{2k}.
    \end{equation}
    Since $F(x,y^*) = I(x),$ we must have $F(x_k,y^*) \le I(x) + 1/(2k)$ and so in particular $x_k\in \mathcal{G}_B.$ Also, $F(x_k, y_{x_k,B}) \le I(x) + 1/(2k)$ by minimality of $y_{x_k,B}$. So, \eqref{eq:xkDef} and minimality conditions of $y_{x_k,B}$ and $y^*$ imply
    \begin{align*}
        F(x,y_{x_k,B}) &\le F(x_k, y_{x_k,B}) + \frac{1}{2k}\\
        & \le F(x_k, y^*) + \frac{1}{2k}\\
        & \le F(x,y^*) + \frac{1}{k}\\
        & \le F(x,y') + \frac{1}{k}
    \end{align*}
    for all $y'\in Y.$ We conclude that $x$ is in the right-hand side of \eqref{eq:UpsiMeasurable} because there is $B\in \mathcal{A}_U$ such that for all $k\in \N$ there is $x_k\in \mathcal{G}_\mathcal{B}$ such that $F(x,y_{x_k,B}) \le F(x,y')+1/k$ for all $y'\in Y_{\mathcal{G},\mathcal{A}}.$ This implies the forward inclusion in \eqref{eq:UpsiMeasurable}. Note that the sequence $y_k = y_{x_k,B},$ $k\in \N$, satisfies 
    \begin{equation}\label{eq:ykSequence}
        \lim_{k\to \infty}F(x,y_k) = I(x).
    \end{equation}

    Now suppose 
    \[x\in \bigcup_{B\in \mathcal{A}_U}\bigcap_{k=1}^\infty\bigcup_{x\in \mathcal{G}_B}\bigcap_{y'\in Y_{\mathcal{G},\mathcal{A}}} C(k,y_{x,B}, y').\]
    We wish to show that there exists a $y^*\in U$ such that $F(x,y^*) = I(x).$ Let $B\in \mathcal{A}_U$ be such that for all $k\in \N$, there is $x_k\in \mathcal{G}_B$ such that 
    \begin{equation}\label{eq:FxkInequality}
        F(x,y_{x_k,B})\le F(x,y')+1/k
    \end{equation}
    for all $y'\in Y_{\mathcal{G},\mathcal{A}}.$ 
        
    We first claim that there is $y'\in Y_{\mathcal{G},\mathcal{A}}$ such that $F(x,y') < \infty.$ Indeed, $I(x)< \infty$ by assumption, and there is a sequence $(y_k)_{k\in \N}\subset Y_{\mathcal{G},\mathcal{A}}$ satisfying \eqref{eq:FxkInequality} and so there exists such a $y'\in Y_{\mathcal{G},\mathcal{A}}.$ Since $x\in C(k,y_{x_k,B}, y')$, it follows that $F(x,y_{x_k,B}) \le R$ for all $k\in \N$, where $R := F(x,y') + 1$.

    The set $F^{-1}(x;R)\cap \overline{B}$ is weakly sequentially compact, and so there is a subsequence of $(y_{x_k,B})_{k\in \N}$ that weakly converges to some $y^*\in \overline{B}\subset U.$ By lower semicontinuity of $F(x,\cdot)$ and \eqref{eq:FxkInequality}, we have $F(x,y^*) \le F(x,y')$ for all $y'\in Y_{\mathcal{G},\mathcal{A}}.$ Because there exists a sequence $(y_k)_{k\in \N}$ in $Y_{\mathcal{G},\mathcal{A}}$ such that \eqref{eq:ykSequence} holds, $F(x,y^*) = I(x).$ Since $y^*\in U$ is follows that $x\in U_{\Psi^{-1}}.$ Thus, the reverse inclusion in \eqref{eq:UpsiMeasurable} is proven, and we may conclude that $U_{\Psi^{-1}}$ is measurable for all open sets $U\subset Y.$

    The conclusions of Kuratowski--Ryll-Nardzewski Selection Theorem are satisfied and so we may conclude the existence of the measurable selection $f$ and the lemma.
\end{proof}

\section{Proof of Theorem~\ref{th:conditions-hold-for-Riemannian-models} } \label{sec:appendix}

\begin{proof}[Proof of Lemma~\ref{lem:XiProperties}]
   Let $u_1,\dots, u_{d-1}$ be an orthogonal basis for the subspace $H$ orthogonal to $v$ and define the change-of-basis matrix $\mathbf{H} = \begin{bmatrix}
       u_1 &\dots & u_{d-1} & v
   \end{bmatrix}.$
   
   Since $w-v \in H$, $w = v + \sum_{i=1}^{d-1}a_i u_i$ for some scalars $a_1,\dots, a_{d-1}.$ If $x \in H,$ then $\Xi_{v\to w}x = x,$ and so $\Xi_{v\to w}$ acts as the identity on $H$. Also, $\Xi_{v\to w} v = w.$ It follows that in the basis $\{u_1,\dots, u_{d-1}, v\}$ the matrix $\Xi_{v\to w}$ is 
   \begin{equation}
       M = \begin{bmatrix}
           1      & 0 & 0     &\dots & a_1\\ 
           0      & 1 & 0     &\dots & a_2\\
           \vdots &   & \ddots &     & \vdots \\
               &   &        &  1  &   a_{d-1} \\ 
           0      &   &    \dots    &  0   & 1
       \end{bmatrix}.
   \end{equation}
    The matrix $M$ has determinant one, and since $\Xi_{v\to w}$ is similar to $M$, so does $\Xi_{v\to w}$.

    The above analysis implies additionally that $\Xi_{v\to w}$ is a nondegenerate linear map, hence the induced map on $\SP$ is bijective. Additionally, since $\Xi_{v\to w}T v = Tw$ for all $T>0$, $\Xi_{v\to w}$ restricted to $\SP_{0,Tv,*}$ is a bijective map from $\SP_{0,Tv,*}$ to $\SP_{0,Tw,*}$.
\end{proof}

\begin{proof}[Proof of \Cref{th:conditions-hold-for-Riemannian-models}]
    We will work out the details fully for \Cref{eg:FPPexample1} and give a sketch for \Cref{eg:FPPexample2}. Now take $g$ to be the function described in \Cref{eg:FPPexample1}.
    We first check condition \ref{stationaryPhiCondition}. The transformations $(\theta^x_*)_{x\in \R^d}$ defined in \Cref{sec:examples-Riemannian} are ergodic by ergodicity of marked Poisson processes with respect to spatial shifts (see \cite{Kingman:MR1207584} or \cite{Daley:MR1950431}). Now let $r\in \R^d.$ From the definition of $\theta^r_*$ in this model:
    \begin{align*}
        g_{\theta^r x, \theta^r_* \omega} & = \int \varphi(\theta^r x - \theta^r y)\mathbf{N}(dy,d\varphi) + \lambda I\\
        & = \int \varphi(x + r - y - r)\mathbf{N}(dy,d\varphi)  + \lambda I\\
        & = g_{x,\omega},
    \end{align*}
    and so condition \ref{stationaryPhiCondition} is satisfied.

    We now establish condition \ref{phiwDiffCondition}. We set $\delta = 1$. The fact that $\Xi_{v\to w}^*$ is measure preserving for all $w\in v+H$ follows from \Cref{lem:XiProperties}. By the uniform compact support requirement, we can find $R>0$ such that if $|y| > R$ and $|w-v|\le 1$, then $\Qq\{\varphi(\Xi_{v\to w} y) = 0\} = 1$. Thus,
\[\|g^{w,v}_x\| \le \lambda + \int \|\varphi(\Xi_{v\to w}(x-y))\| \mathbf{N}(dy,d\varphi)\le \lambda + \int\|\varphi\|_{C^2}\1_{|x-y|\le R}\mathbf{N}(dy,d\varphi).\]
Also, there is $C_1 > 0$ such that for all $w$ satisfying $|w-v|\le 1$,
\begin{align*}
    \|\partial_{w_i} g^{w,v}_x\| &\le \int \|\partial_{w_i}\varphi(\Xi_{w}(x-y))\|\mathbf{N}(dy,d\varphi)\\
    & \le \int \|\nabla \varphi(\Xi_{v\to w}(x-y))\|\|\partial_{w_i}\Xi_{v\to w}\| |x-y|\mathbf{N}(dy,d\varphi) \\
    & \le C_1 \int \|\varphi\|_{C^2}\1_{|x-y|\le R}\mathbf{N}(dy,d\varphi).
\end{align*}
Similarly, there is $C_2 > 0$ such that for all $w$ satisfying $|w-v|\le 1$,
\[\|\partial_{w_j }\partial_{w_i} g^{w,v}_x\| \le  C_2 \int \|\varphi\|_{C^2}\1_{|x-y|\le R}\mathbf{N}(dy,d\varphi).\]
Let $\eta$ be a smooth function whose support is contained by the ball centered at the origin of radius $2R$ such that $\eta(x) \ge \1_{|x|\le R}$ for all $x\in \R^d.$ For a sufficiently large constant $C$ random function $\YField$ defined by
\begin{equation}\label{eq:YSum}
    \YField(x) = \lambda + C \int \|\varphi\|_{C^2}\eta(x-y)\mathbf{N}(dy,d\varphi)
\end{equation}
satisfies \eqref{eq:boundForg} and \ref{boundedY}. Indeed, $\YField$ is stationary with respect to lattice shifts by stationarity of Poisson points, verifying \ref{Ystationary}. In addition, $\YField$ has finite range dependence due to the compact support of $\eta$, and so \ref{YFiniteRange} follows. Now we verify the moment condition in \ref{YMomentCondition}. If $(\varphi_i)_{i\in \N}$ are and i.i.d. family with distribution $\Qq$, and $N(R)$ the number of Poisson points in a ball of radius $2R,$ then, for $\ell > 0$ such that $\Qq\|\varphi_1\|^\ell < \infty$, the Marcinkiewicz--Zygmund inequality (\cite{Marcinkiewicz1937}) implies, for some $C,C'>0,$
\[\E\Big[\sup_{x\in [0,1]^{d}} |\YField(x)|^{\ell}\Big] \le C + C \E\Big[\Big|\sum_{i=1}^{N(R)}\|\varphi_i\|\Big|^\ell\Big] \le C + C'\E| N(R)|^{\ell/2} < \infty.\]
Because \Cref{eg:FPPexample1} assumes that $\Qq\|\varphi_1\|^{\beta} < \infty$ for some $\beta>4d$, the above display implies \ref{YMomentCondition} of \Cref{eg:FPPexample1} for the same $N$.

The uniform positive definite condition in \ref{uniformPositiveDefinite} is satisfied due to the $\lambda I$ factor in \eqref{eq:phiSumRepresentation}.

This completes the proof in the case where $g$ is as in \Cref{eg:FPPexample1}.

Now take $g$ as given in \Cref{eg:FPPexample2}. The only meaningful difference to the preceding argument is in computing derivatives of $g^{w,v}$ and bounding them by an appropriate field $\YField$. We have 
\[g^{w,v}_x = \exp\Big(\int\varphi(\Xi_{v\to w}(x-y))\mathbf{N}(dy,d\varphi)\Big).\]

Under \Cref{eg:FPPexample2}, $\|\varphi\|_{C^2}$ is bounded by a deterministic constant. We let $\eta$ be the smooth function used previously whose suppose contains the ball of radius $2R$. Then, it suffices to take
\begin{equation}\label{eq:YProduct}
    \YField(x) = \exp\Big( C\int \eta(x-y) \mathbf{N}(dy,d\varphi)\Big)+ C
\end{equation} 
for a sufficiently large constant $C$. Specifically, there is a constant $C>0$ such that 
\begin{equation}\label{eq:productUpperBound}
    \|\exp\Big(\int\varphi(\Xi_{v\to w}(x-y))\mathbf{N}(dy,d\varphi)\Big)\|_{C^2,x} \le Y(x)
\end{equation}
holds 
The verification that \eqref{eq:YProduct} satisfies \ref{boundedY} is similar to the argument for \eqref{eq:YSum}. Indeed, $\log \YField$ is of the same sum form as in \eqref{eq:YSum}, and so the stationarity and finite range dependence conditions follow in the same manner. Additionally, $\E|\sup_{x\in [0,1]^d}Y(x)|^\ell < \infty$ for all $\ell > 0$ by compact support of $\eta$ and the fact that Poisson random variables have finite moment of all orders.

Now we will sketch the argument for \eqref{eq:productUpperBound}. For a path $X(t)$ in matrix space, 
\[\frac{d}{dt} \exp(X(t)) = \int_0^1 e^{\alpha X(t)}\frac{d X(t)}{dt}e^{(1-\alpha)X(t)}d\alpha\]
(see Theorem 2.19 in Chapter IX of \cite{Kato:1966:PTL}). Also, for a matrix $M$, we have $\|e^{M}\|\le e^{\|M\|}.$ It follows that 
\begin{align*}
   \|\frac{d}{dt}e^{X(t)}\| \le e^{\|X(t)\|}\|\frac{d}{dt}X(t)\|.
\end{align*}
The above formula and the fact almost surely $\|\varphi\|\le C_1$ and $\|\Xi_{v\to w}\|\le C_2$ for deterministic constants $C_1,C_2$ can be used to show that $\partial_{w_i}\exp\Big(\int\varphi(\Xi_{v\to w}(x-y))\mathbf{N}(dy,d\varphi)\Big)$ and $\partial_{w_j}\partial_{w_i}\exp\Big(\int\varphi(\Xi_{v\to w}(x-y))\mathbf{N}(dy,d\varphi)\Big)$ are bounded by \eqref{eq:YProduct} for a sufficiently large deterministic constant $C$.

Finally, the relation $\int \varphi(x-y)\mathbf{N}(dy,d\varphi) \succeq 0$ implies $g_x \succeq I$ and so the uniform positive definite condition in \ref{uniformPositiveDefinite} holds with $\lambda = 1$.

\end{proof}

\section{Proof of Theorem~\ref{thm:main-on-broken-lines}.}\label{sec:proofs_broken-lines}
\subsection{Checking conditions \ref{cond:subbadd} -- \ref{cond:for-finiteness}, \ref{cond:setup_diff} -- \ref{cond:M_infty}}

For $r>0$, distinct $x,y\in\R^d$ and $n\in\N$, and $\omega\in\Omega$ we define  $\NicePaths^r_{x,y,n}(\omega)$ to be the set of paths 
$\gamma\in\Paths_{x,y,n}$ satisfying the following condition:
for all  $i=0,1,\ldots,n-1$, $|\gamma_{i+1}-\gamma_{i}|\le r$, and 
there are numbers $k$ and $i_0,i_1,\ldots, i_k$ such that 
\begin{enumerate}[label = (\roman*), ref=\theenumi{}\rm{\roman*}]
\item $0=i_0<i_1<\ldots<i_k=n$; 
\item $\gamma_{i_j}\in\omega$ for $j=1,\ldots,k-1$;
\item $\gamma_{i_j}\ne \gamma_{i_m}$ if $j\ne m$;
\item for each $j=0,1,\ldots,k-1$, and every $i\in\{i_j,i_j+1,\ldots,i_{j+1}\},$
\[
    \gamma_{i}=
    \frac{i-i_j}{i_{j+1}-i_j} \gamma_{i_{j+1}}+\frac{i_{j+1}-i}{i_{j+1}-i_j} 
    \gamma'_{i_{j}}.
\]
\end{enumerate}

We also use notation $\NicePaths^r_{*,*,n}(\omega)$, $\NicePaths^r_{*,*,*}(\omega)$, etc., similarly to \eqref{eq:paths_stars}.

\begin{lemma}
    \label{lem:existence-minimizers}
    There a set $\Omega'\in\Fc$ with $\Prb(\Omega')=1$, a number $r>0$, and a jointly measurable map 
    \begin{align*}
        \gamma:\Omega\times\R^d\times\R^d &\to \Paths\\
                (\omega, x,y) &\mapsto \gamma_\omega(x,y)
    \end{align*}
    such that 
    for all $\omega\in\Omega'$ and $x,y\in \R^d$, $\gamma_\omega(x,y)\in\NicePaths^r_{x,y,*}(\omega)$
    and it is  a geodesic under $\omega$.
\end{lemma}

We prove this lemma in Section~\ref{sec:broken-line-minimizers}.

To apply our general theorems to this model, we need to interpret the action as a function of continuous paths from $\SP$. We will define $A_\omega$
separately on broken line paths and other paths. Namely, we set
\[
    A_\omega(\psi)=A_\omega(\psi_0,\psi_{1},\ldots\psi_{n})
\] 
if $n\in\N$ and $\psi\in\SP_{*,*,n}$ satisfies $\psi_{k+t}=(1-t)\psi_k+t\psi_{k+1}$ for all $k=0,1\ldots,n-1$
and $t\in[0,1]$. If a path $\psi\in \SP$ is not of this form, we set $A_\omega(\psi)=+\infty$. There is a natural bimeasurable bijection between discrete paths in $\Paths$ and finite action paths in $\SP$. In particular, Lemma~\ref{lem:existence-minimizers} automatically
provides a measurable representation of continuous optimal paths $\gamma\in\SP$:
for $x,y\in\R^d$ and almost all
$\omega\in\Omega$, there is an action minimizer from~$\SP$  traveling from $x$ to $y$ and switching directions at finitely many Poissonian points. These Poissonian points and the endpoints $x$ and $y$  will be called binding vertices. 

\medskip

With this continuous path interpretation at hand, we can check conditions 
\ref{cond:subbadd} -- \ref{cond:for-finiteness} of Section~\ref{sec:generalSetUp}.
Condition~\ref{cond:subbadd} is a corollary
of \eqref{eq:concatenation}. Condition~\ref{cond:skew-invariance} follows directly from the definition of action in~\eqref
{eq:Poisson-action} and the identity $F_{\omega}(x) = F_{\theta_*^{-x}}(0)$. Condition~\ref{cond:measurability} is implied by 
Lemma~\ref{lem:existence-minimizers}. 
Conditions~\ref{cond:cone} and~\ref{cond:for-finiteness} with $\Cone=\R^d$ follow from
 \eqref{eq:bound_on_rho_1}.

\medskip

 Let us now check conditions \ref{cond:setup_diff} and \ref{cond:M_infty}.

 Fixing an arbitrary $v\in\R^d\setminus\{0\}$, we define $H$ as the orthogonal complement to the line spanned by $v$. The family of transformations
  $(\Xi_{v\to w})_{w\in v+ H}$ of $\R^d$ is defined by~\eqref{eq:XiFPP}.
 Also, for $w\in v+H$, we define the transformation $\Xi^*_{v\to w}$ of $\Omega$ as the pushforward of $\omega\in\Omega$ by $\Xi_{v\to w}$. In other words, we apply the transformation
 $\Xi_{v\to w}$ to each Poissonian point. Choosing $\delta=1$, we see that the setup requirement of~\ref{cond:setup_diff}
 holds.  It remains to check \ref{cond:M_infty}. 
 Let us first compute for all $y\in\R^d$:
 \begin{align*}
     \partial_{w_k} L(\Xi_{v\to w}y)&=\sum_{j=1}^d \partial_j L(\Xi_{v\to w}y)\partial_{w_k} (\Xi_{v\to w}y)_j 
     \\&= \sum_{j=1}^d \partial_j L(\Xi_{v\to w}y)\frac{\langle v,y\rangle}{|v|^2 }\delta_{kj}=\partial_k L(\Xi_{v\to w}y)\frac{\langle v,y\rangle}{|v|^2 },\quad k=1,\dots,d,
 \end{align*}
 and
 \begin{align*}
     \partial_{w_kw_j} L(\Xi_{v\to w}y)=\partial_{kj} L(\Xi_{v\to w}y)\frac{\langle v,y\rangle^2}{|v|^4},\quad k,j=1,\dots,d.
 \end{align*}
 Therefore, if $\gamma^T(v)\in\NicePaths^r_{*,*,n}(\omega)$, then, since
 \begin{align*}
    B(w,v,\gamma^T(v))=&\sum_{i=0}^{n-1} L(\Delta_i\Xi_{v\to w} \gamma^{T}(v))
    \\&\qquad \qquad\qquad +\frac{1}{2}\sum_{i=0}^{n-1} (F_{\Xi^*_{w\to v}\omega}(\Xi_{v\to w}\gamma_i)+F_{\Xi^*_{w\to v}\omega}(\Xi_{v\to w}\gamma_{i+1}))
    \\
    =&\sum_{i=0}^{n-1} L(\Xi_{v\to w} \Delta_i\gamma^{T}(v))
    +\frac{1}{2}\sum_{i=0}^{n-1} (F_{\omega}(\gamma_i)+F_{\omega}(\gamma_{i+1})),
\end{align*} 
 we have
 \begin{equation}\label{eq:firstDerivB_BrokenLine}
    \partial_{w_j} B(w,v,\gamma^T(v)) = \frac{1}{|v|^2}\sum_{i=0}^{n-1}\partial_{j}L(\Xi_{v\to w}\Delta_i\gamma^T(v))\langle v,\Delta_i\gamma^T(v)\rangle
 \end{equation}
 and
 \begin{align*}
     \partial_{w_kw_j} B(w,v,\gamma^{T}(v)))&=\sum_{i=0}^{n-1} \partial_{w_kw_j}L(\Xi_{v\to w}\Delta_i \gamma^{T}(v))
     \\ & =\frac{1}{|v|^4}\sum_{i=0}^{n-1} \partial_{kj}L(\Xi_{v\to w}\Delta_i \gamma^T(v))\langle v, \Xi_{v\to w}\Delta_i \gamma^{T}(v)\rangle^2.
 \end{align*}
 Since  $L\in C^2(\R)$ and the increments of $\gamma^T(v)$ are bounded by $r$, we  obtain that there is a number $D=D(v)$ such that 
 if $|w-v|<1$, then
 \begin{align}\label{eq:estimating-2nd-derivatives}
    | \partial_{w_kw_j} B(w,v,\gamma^{T}(v)))|&\le D(v) \sum_{i=0}^{n-1} |\Delta_i \gamma^{T}(v)|^2.
 \end{align}
Using \ref{asm:lower-L-at-0} we obtain that there is $c(r)>0$ such that if $|y|\le r$, then
\[
    L(y)>c(r) |y|^2.
\]
Therefore, we can extend \eqref{eq:estimating-2nd-derivatives}:
\begin{align*}
    | \partial_{w_kw_j} B_\omega(w,v,\gamma^{T}(v)))|&\le c^{-1}(r) D(v) \sum_{i=0}^{n-1} L(\Delta_i \gamma^{T}(v))
    \\ &\le c^{-1}(r) D(v) A_\omega(\gamma^T(v))
 \end{align*}
and \ref{cond:M_infty} follows since 
$\limsup_{T\to\infty} (A(\gamma^T(v))/T)=\Lambda(v)$.
 
The expression for  $\nabla \Lambda$ in~\eqref{eq:BrokenLineDerivative} follows from \eqref{eq:firstDerivB_BrokenLine} and \eqref{eq:H-diff-of Lambda}.

\subsection{Proof of Lemma~\ref{lem:existence-minimizers}}\label{sec:broken-line-minimizers}
We will need several auxiliary lemmas first.

For every $x,y\in\R^d$ and all $n\in\N$, we define  $\gamma(x,y,n)\in\Paths_{x,y,n}$ by
\begin{equation}
    \label{eq:evenly_spaced}    
\gamma_k(x,y,n)=\frac{k}{n}y+\Big(1-\frac{k}{n}\Big)x,\quad k=0,\ldots,n.
\end{equation}

\begin{lemma} \label{lem:nice-paths}
    \begin{enumerate}[label=\arabic*., ref = \rm{\arabic*}]
        \item \label{it:geodesic-restriction}  For all $\omega\in\Omega$, 
    if $\gamma=(\gamma_0,\gamma_1,\ldots,\gamma_n)\in\Paths_{*,*,n}$ is a geodesic, then so is
    $(\gamma_i,\gamma_{i+1},\ldots,\gamma_k )\in\Paths_{*,*,k-i}$ for all $i,k$ satisfying
    $0\le i< k \le n$. 

    \item \label{it:even-steps} For all $\omega\in\Omega$,  $x,y\in\R^d$, $n\in\N$, all
     $\gamma\in \Paths_{x,y,n}$, if $\gamma_k\notin \omega$ for all $k=1,\ldots,n-1$,
     then
     \[
        A_\omega(\gamma)\ge A_\omega(\gamma(x,y,n)),
     \]
     where $\gamma(x,y,n)$ is defined in~\eqref{eq:evenly_spaced}.  
     \item  \label{it:bdd-steps}   There is $r>1$ such that
     if $\gamma$ is a geodesic for some $\omega\in\Omega$, then the distance between any consecutive points of $\gamma$ is bounded by $r$.
 
     \item \label{it:path-struct} 
     Let $r$ be the number provided in part \ref{it:bdd-steps}.
     For all $\omega\in\Omega$, all distinct $x,y\in\R^d$, all $n\in\N$, and
every path $\gamma\in \Paths_{x,y,*}$, there is~$\gamma'\in \NicePaths^r_{x,y,*}(\omega)$
satisfying $A_\omega(\gamma')\le A_\omega(\gamma)$.
\item \label{it:best-action-via-Q} For all distinct $x,y\in\R^d$ and all $\omega\in\Omega$,
\begin{equation}
    \label{eq:def_random_metric3}
\Amin_\omega(x,y)=\inf_{\gamma\in\NicePaths^r_{x,y,*}(\omega)} A_\omega(\gamma).  
\end{equation}
\end{enumerate}
\end{lemma}
\begin{proof}
Part~\ref{it:geodesic-restriction} is obvious.

To prove part \ref{it:even-steps}, it suffices to note that since $\sum_i \Delta_i\gamma_i=y-x$, convexity of $L$ implies
\begin{align*}
    \frac{1}{n}(A_\omega(\gamma)- A_\omega(\gamma(x,y,n)))
    &\ge\frac{1}{n}\sum_{i=0}^{n-1} L(\Delta_i \gamma) -L\Big(\frac{y-x}{n}\Big)\ge 0.
\end{align*}

To prove part~\ref{it:bdd-steps}, we need to find $r$
such that if $|x-y|>r$, then there is $n\ge 2$ such that $A_\omega(\gamma(x,y,n))< A_\omega(x,y)$ (here $(x,y)\in\Paths_{x,y,1}$). It suffices to check that
\begin{equation}
    \label{eq:proving-points-are-not-far}
   n L((y-x)/n)+n < L(y-x) 
\end{equation}
for some $n\ge 2$. 
Let $L^*=\sup_{|x|\le 2} L(x)<\infty$.
We can use the superlinearity condition
\ref{asm:superlinear_L}
to pick $r>2$ such that $|y-x|>r$
implies
$L(y-x)>(L^*+1)|y-x|$.

If $|y-x|>r$, we set $n=\lfloor |y-x| \rfloor$. Then $n\ge 2$. In addition,
$|y-x|/n\le 2$ implies
\begin{align*}
    nL((y-x)/n)+n\le |y-x| L^*+|y-x|< L(y-x),
\end{align*}
i.e., \eqref{eq:proving-points-are-not-far} holds.

Part~\ref{it:path-struct} 
follows from parts~\ref{it:even-steps} and~\ref{it:bdd-steps}. 
Part~\ref{it:best-action-via-Q} follows from part~\eqref{it:path-struct} 
\end{proof}

 \begin{lemma} For all $x,y\in\R^d$, $\Amin_\omega(x,y)$ is a random variable. 
\end{lemma}
\begin{proof} Due to~\eqref{eq:def_random_metric3}, we can write
\[
    \Amin_\omega(x,y)=\lim_{m\to\infty} \Amin_\omega(x,y,m),
\]
where
\[
    \Amin_\omega(x,y,m)=\min_{\substack{\gamma\in\NicePaths^r_{x,y,*}(\omega)\\ \gamma\subset\Ball(0,m)}} A_\omega(\gamma).
    \]
Since to compute $\Amin_\omega(x,y,m)$ we just need to  search through finitely many paths defined by Poisson points they pass through,
$\Amin_\omega(x,y,m)$ is a random variable for a fixed $m$.  Therefore, the limit as $m\to\infty$, $\Amin_\omega(x,y)$, 
is also a random variable. 
\end{proof}

Using part~\ref{it:even-steps} of Lemma~\ref{lem:nice-paths}, we can define distances
between any two points along a straight line:
\begin{align}
    \label{eq:distance_along_straight}
    \strdist_\omega(x,y)&=\inf_{n\in\N} A_\omega(\gamma(x,y,n))
    \\ \notag
    &= \inf_{n\in\N} \Big(n L\Big(\frac{x-y}{n}\Big)+\sum_{i=1}^{n-1} F_\omega(\gamma_i(x,y,n))\Big)+\frac{1}{2}(F_\omega(x)+F_\omega(y)).
\end{align}
We also define $\strdist_\omega(x,x)=0$ for all $x\in\R^d$.
Let $\bar \Omega$ be the set of all $\omega$ such that no three points of $\omega$ are on the same straight line.
Then $\Prb(\bar \Omega)=1$.

\begin{lemma}\label{lem:factsAboutRho} 
    Let $\omega\in\bar\Omega$. Then 
        for all distinct $x,y\in\R^d$, 
        $\strdist_\omega(x,y)<\infty$, and  for all compact sets $K\subset \R^d$, 
        \[\inf_{x\in K,\,y\in \R^d,y\ne x}\rho_\omega(x,y) > 0.\]
\end{lemma}

\begin{proof} 
The  upper bound
is trivial. 
Now we prove the lower bound. Let $\delta\in (0,1)$ be less than the minimum distance between any Poisson point $p_1\in K\cap \omega$ and any other Poisson point $p_2\in \omega\setminus\{p_1\}.$ Then, if $x\in K$, $y\ne x$ and $n\in \N$ satisfy $|y-x|/n<\delta$, then either $x$ or $\gamma_1(x,y,n)$ are not a Poisson point and so $A_\omega(\gamma(x,y,n))>\frac{1}{2}$. If, rather, $|y-x|/n\ge \delta$, then, due to \ref{asm:lower-L-at-0},
 $A_\omega(\gamma(x,y,n))\ge n L (|y-x|/n)\ge c\delta^2$. Thus, the infimum in question is bounded below by $1/2 \wedge (c\delta^2)$.
\end{proof}

 We recall that $\ast$-connected sets are defined 
in~\Cref{sec:proofs_of_lemmas}.

For a random field $(X_k)_{k\in\Z^d}$ and a set $U\in \Z^d$, we denote
\[
    X(U)=\sum_{k\in U}X_k.
\]

\begin{lemma}[\cite{LaGatta_Wehr2010}]\label{lem:LaGatta-Wehr} Let $(X_k)_{k\in\Z^d}$ be a stationary random field with finite dependence range. Suppose
     that
    \begin{equation}
        \label{eq:no_atom_at_0}
     \Prb\{X_0=0\}=0.
    \end{equation} 
     Then, there is $\beta>0$ such that 
    for all $A>0$, the following holds. With probability~$1$, 
    there is $N$ such that for $n\ge N$, 
    if $\Gamma$ is a $\ast$-connected subset of $\Z$ containing $0\in\Z^d$
    and $X(\Gamma)\le An$, then $|\Gamma|\le \beta An$.    
\end{lemma}

\begin{remark} Lemma 2.2 was stated slightly differently in \cite{LaGatta_Wehr2010}. We replace a condition 
    on the atom mass at $0$ by the stronger no atom at $0$
    condition in \eqref{eq:no_atom_at_0}. We also replace the condition on sets $\Gamma$ that  was allowed to vary with $n$ with a stricter requirement independent of~$n$. In \cite{LaGatta_Wehr2010}, only an estimate $|\Gamma|\le Bn$ is stated as the conclusion of the lemma but it follows from the proof that $B$ can be taken in the form of $\beta A$, where $\beta$ only depends on the distribution of $X_0$.
\end{remark}

\bigskip

We will use the notation $A+B=\{x+y: x\in A,\ y\in B\}$ for $A,B\subset\R^d$.
Let's fix $R>2r$ and for $k\in\Z$ define
\[
    \Icube_k=Rk+[0,R]^d,
\]   
\[
    \Icube_k^+=\Icube_k+[-r,r]^d,
\]
\[
    \Delta \Icube_k =  \Icube_k^+ \setminus \Icube_k,
 \]   
\[
    \Jcube_k=\Icube_k+[-R,R]^d,    
\]
\[   
\Jcube^-_k= \Icube_k+[-R+r,R-r]^d,
\]
\[
    \Delta \Jcube_k =  \Jcube_k \setminus \Jcube^{-}_k.
 \]

\begin{lemma} \label{lem:lower-bound-in-R-cube}
   There is a stationary $(0,\infty)$-valued finite dependence range random field $(\xi_k)_{k\in \Z^d}$  such that
    if $k\in\Z^d,$ $n\in\N$, and $\gamma\in \NicePaths^r_{*,*,n}(\omega)$ is contained entirely in $\Jcube_k$ and satisfies $\gamma_0\in\Delta\Icube_k$, $\gamma_n\in \Delta\Jcube_k$,
    then 
    \begin{equation}
        \label{eq:lower-bound-in-1-cube}
        A_\omega(\gamma)\ge \xi_k(\omega),\quad  \as
    \end{equation}
\end{lemma}
\begin{proof}
    For $k\in \Z^d$ let $P_k$ denote the set of Poisson points in $\Jcube_k$ and define the random variable
\[\xi_k(\omega) = \inf\{\rho_\omega(x,y)\,:\,x\in \Delta\Icube_k,\  y\in P_k\cup \Delta \Jcube_k,\ y\neq x\}.\]
First, note that the collection $(\xi_k)_{k\in \Z^d}$ is stationary by stationarity of the Poisson process. Since each $\xi_k$ is a function of Poisson points contained in $\Jcube_k$, a bounded set,  the collection $(\xi_k)_{k\in \Z^d}$ has finite range of dependence. Additionally, 
\Cref{lem:factsAboutRho} implies that $\xi_k > 0$ almost surely.

We claim that $A_\omega(\gamma) \ge \xi_k(\omega)$ for all $\gamma\in \NicePaths^r_{*,*,n}$ satisfying the conditions in the lemma.  Note that $\Delta \Icube_k \cap \Delta \Jcube_k = \emptyset.$ As a consequence, there exists an $i^*\in \{1,\dots, n\}$ such that $\gamma_{i^*}\in P_k\cup \Delta \Jcube$ and $\gamma_j\notin P_k\cup \Delta \Jcube_k$ for all $j\in \{0,\dots, i^*-1\}.$ Then, $A_\omega(\gamma) \ge A_\omega(\gamma_0,\dots,\gamma_{i^*}).$ Finally, part~\ref{it:even-steps} of \Cref{lem:nice-paths} implies that
\[A_\omega(\gamma_0,\dots,\gamma_{i^*}) \ge \rho_\omega(\gamma_0,\gamma_{i^*}),\]
and the right-hand side is bounded below by $\xi_k(\omega).$
\end{proof}

\begin{lemma}\label{lem:action_at_least_linear}
    There is $C>0$ such that with probability 1, there is $D>0$ such that if $x\in I_0$,  $|y|>D$ and $\gamma\in\NicePaths^r_{x,y,*}(\omega)$, then 
    $A_\omega(\gamma)>C|y|$. 
\end{lemma}
Once the existence of the shape function $\Lambda$ is established, 
it follows from this lemma that $\Lambda(v)>0$ for all $v\ne 0$, which, 
according to Theorem~\ref{thm:limit-shape-diff},
implies that the boundary of the limit shape is diffeomorphic to a sphere.  

\begin{proof} 
Assume that no $C$ described in the statement exists. It means that with positive probability,
for every $\eps>0$, there 
are sequences $n_m\in\N$, $x_m\in I_0$, $y_m\in\R^d$,  $\gamma^m\in \NicePaths^r_{x_m,y_m,n_m}$ 
with $A_\omega(\gamma^m)<\eps |y_m|$ and 
$|y_m|\to \infty$. 
Recalling that $\beta$ is the constant provided by Lemma~\ref{lem:LaGatta-Wehr} and choosing $\eps$ to satisfy
\begin{equation}
    \label{eq:choosing_eps}
    0<\eps< (2\sqrt{d}R\beta)^{-1},
\end{equation}
we will arrive at a contradiction. 

We are going to decompose $\gamma^m$ into smaller pieces. We will use the fact that the increments of $\gamma^m$ are bounded by $r$.
First, we set $k_0=0\in\Z^d$ and 
\[i_0=\min\{s\in\N:\ \gamma_{s}\in \Delta\Icube_0  \}.
\] 
Then, inductively, for $j=0,1,2,\ldots$, we define
\begin{align*}
    i_{j+1}&=\min\{s> i_j:\ \gamma_{s}\notin \Jcube_{k_{j}}\}\wedge n_m,
\end{align*}
and choose $k_{j+1}\in\Z^d$ so that $k_{j+1}-k_j\in \{-1,0,1\}^d$ and
$\gamma_{i_{j+1}}\in \Icube^+_{k_{j+1}}$. The latter can always be accomplished since the distance between two consecutive vertices of $\gamma$ is bounded by $r$. The same argument implies $\gamma_{i_{j+1}-1}\in \Delta\Jcube_{k_{j}}$. We 
define 
$N_m=\min \{j:\ i_j=n_m\}$ and 
\[
    \gamma^{m,j}=(\gamma^m_{i_j},\gamma^m_{i_j+1},\ldots, \gamma^m_{i_{j+1}-1}), \quad j=0,\ldots,N_m-1.
\]
These paths satisfy the conditions of Lemma~\ref{lem:lower-bound-in-R-cube},  Hence, for $\Gamma_m=\{k_0,\ldots,k_{N_m-1}\}$, we can use the random field  $(\xi_k)_{k\in \Z^d}$ provided by
Lemma~\ref{lem:lower-bound-in-R-cube} to obtain
 \[
    \xi(\Gamma_m)=\sum_{k\in\Gamma_m} \xi_k\le \sum_{j=0}^{N_m-1} A_\omega(\gamma^{m,j})\le A_\omega(\gamma^m)\le \eps |y_m|\le 
    \eps \lceil |y_m|\rceil.
\] 

Since $\Gamma_m$ is $\ast$-connected, we can 
apply Lemma~\ref{lem:LaGatta-Wehr}. Choosing $A=\eps$
and using~\eqref{eq:choosing_eps}, we obtain 
for sufficiently large $m$,
\begin{equation}
    \label{eq:Gamma-size}
    |\Gamma_m|\le \beta \eps \lceil |y_m|\rceil < 
\frac{1}{2\sqrt{d}R} \lceil |y_m|\rceil.
\end{equation}
But $\Gamma_m$ a $\ast$-connected set containing both $0$ and 
$k_{N_m-1}$. 
Since  $|y_m-Rk_{N_m-1}|\le 2R$ and $|Rk_{N_m-1}|\le |\Gamma_m|\sqrt{d}R$, we obtain $|y_m|\le 2R+ |\Gamma_m|\sqrt{d}R$
contradicting \eqref{eq:Gamma-size} and completing the proof.
\end{proof}

Now we can complete the proof of Lemma~\ref{lem:existence-minimizers}
\begin{proof}
Due to Lemma~\ref{lem:action_at_least_linear}, for almost all $\omega\in\Omega$, the following 
holds for all $R>0$: there is $D_R=D_R(\omega)$ such that 
if $x,y\in \Ball(0,R)$ and  a path 
$\gamma\in \NicePaths^r_{x,y,*}(\omega)$ is not contained in $\Ball(0,D_R)$ then $A_\omega(\gamma)\ge \Amin_\omega(x,y)+1$. 
Thus, paths in $\NicePaths^r_{x,y,*}(\omega)$ with smaller actions 
are contained in $\Ball(0,D_R)$. Since there are finitely many paths in $\NicePaths^r_{x,y,*}(\omega)$ contained in that ball, at least one of them realizes $\Amin_\omega(x,y)$. 
If such path is unique, we set $\gamma_\omega(x,y)$ to be that path.
If there are at least two minimizing paths, we need a tie-breaking rule. For example, 
if there is a minimizer not passing through any Poissonian points, we let $\gamma_\omega(x,y)$ to be that minimizer (it is unique).
If all minimizers pass through some Poissonian points, we choose $\gamma_\omega(x,y)$ to be the one 
containing the Poissonian point with minimal Euclidean norm. On a set of  probability 1, this procedure results in a unique path. We define $\gamma_\omega(x,y)=(x,y)\in\Paths_{x,y,1}$ on the 
complement of this event.

To prove that thus defined geodesic $\gamma$ is measurable, we note that (i) $\gamma_\omega(x,y)$ is the a.s.-limit of action minimizers restricted to the ball $\Ball(0,D)$, as $D\to\infty$; (ii) these restricted minimizers are measurable since they are chosen among finitely many paths.
\end{proof}

\bibliographystyle{alpha} 
\bibliography{Burgers,polymer}

\end{document}